\documentclass{amsart}
\usepackage[english]{babel}
\usepackage[utf8]{inputenc}
\usepackage{amsmath, amssymb,epic,graphicx,mathrsfs,enumerate,stmaryrd}

\usepackage{xcolor}
\definecolor{air}{rgb}{0.36, 0.54, 0.66}

\usepackage[colorlinks=true]{hyperref}

\hypersetup{
	citecolor=air,
	linkcolor=air}

\usepackage{amsthm,booktabs,multirow}

\DeclareMathOperator{\perm}{Sym}
 
\DeclareMathOperator{\frat}{Frat} 
\DeclareMathOperator{\aut}{Aut} 
 
\DeclareMathOperator{\out}{Out} 

\DeclareMathOperator{\Out}{Out} 
\DeclareMathOperator{\soc}{soc}

\DeclareMathOperator{\Hom}{Hom}

\DeclareMathOperator{\gl}{GL} 
 
\DeclareMathOperator{\Z}{Z} 
\DeclareMathOperator{\fit}{Fit} 

\DeclareMathOperator{\C}{C}

\newcommand{\Aut}{\mathrm{Aut}}
\newcommand{\psl}{\mathrm{PSL}}
\newcommand{\al}{\alpha}

\def\t{\tau}
\def\d{\delta}
\def\a{\alpha}

\newcommand{\gen}[1]{\left\langle#1\right\rangle} 

\def\<#1{\langle #1\rangle}

\newcommand{\F}{\mathbb F}

\newcommand{\st}{such that }
\newcommand{\ifa}{if and only if }

\newcommand{\ov}[1]{\overline{#1}}

\newcommand{\ca}{\mathcal}

\newcommand{\nn}{\mathrel{\unlhd}}
\newcommand{\sd}{\rtimes}

\def\psl#1#2{\mathord{\mathrm{PSL}_{#1}(#2)}}

\def\gl#1#2{\mathord{\mathrm{GL}_{#1}(#2)}}

\def\gu#1#2{\mathord{\mathrm{GU}_{#1}(#2)}}

\def\psu#1#2{\mathord{\mathrm{PSU}_{#1}(#2)}}

\newtheorem{thm}{Theorem}
\newtheorem{cor}[thm]{Corollary}
\newtheorem{step}{Step} \newtheorem{lemma}[thm]{Lemma}
\newtheorem{prop}[thm]{Proposition} 
 \newtheorem{defn}[thm]{Definition}

\numberwithin{equation}{section}

\let\lam\lambda
\let\om\omega
\renewcommand{\omega}{\lam}

\renewcommand{\leq}{\leqslant}
\renewcommand{\geq}{\geqslant}
\renewcommand{\nn}{\trianglelefteqslant}

\begin{document}
	\bibliographystyle{amsplain}

\title{Abelian supplements in almost simple groups}

\author{Mauro Costantini}
\address{M. Costantini, Universit\`a di Padova, Dipartimento di Matematica ``Tullio Levi-Civita'', Via Trieste 63, 35121 Padova, Italy}
\email{costantini@math.unipd.it}

\author{Andrea Lucchini}
\address{A. Lucchini, Universit\`a di Padova, Dipartimento di Matematica ``Tullio Levi-Civita'', Via Trieste 63, 35121 Padova, Italy}
\email{lucchini@math.unipd.it}

\author{Daniele Nemmi}
\address{D. Nemmi, Universit\`a di Padova, Dipartimento di Matematica ``Tullio Levi-Civita'', Via Trieste 63, 35121 Padova, Italy}
\email{dnemmi@math.unipd.it}
\thanks{Project funded by the EuropeanUnion – NextGenerationEU under the National Recovery and Resilience Plan (NRRP), Mission 4 Component 2 Investment 1.1 - Call PRIN 2022 No. 104 of February 2, 2022 of Italian Ministry of University and Research; Project 2022PSTWLB (subject area: PE - Physical Sciences and Engineering) " Group Theory and Applications"}
\begin{abstract}
	Let $G$ be a finite almost simple group with socle $G_0$. In this paper we prove that whenever $G/G_0$ is abelian, then there exists an abelian subgroup $A$ of $G$ such that $G=AG_0$. We propose a few applications of this structural property of almost simple groups.
\end{abstract}

\maketitle

\hbox{}

\section{Introduction}

Let $G$ be a finite group and let $N\nn G$. A lot of properties of the group $G$ are inherited by the quotient group $G/N$. On the other hand, the knowledge of $G/N$ and $N$ gives only a partial understanding on $G$. As an example, whenever $H$ is a complement of $N$ in $G$, we know that there exists a subgroup of $G$, $H$, which is isomorphic to $G/N$. An interesting question is whether for a group $G$  and a normal subgroup $N\nn G$, we can find a subgroup $H$ such that $G=HN$ and $H$ preserve some properties of $G/N$, but not necessarily the isomorphism class. An instance of this problem is the following: if $G/N$ belongs to a certain class of groups $\mathcal C$, can we find a subgroup $H$ with $G=HN$ such that $H$ belongs to $\mathcal C$ as well? If $\mathcal C$ respects some properties, the answer is affirmative.

\begin{prop}\label{class}Let $\ca C$ be a class of finite groups which satisfies the following properties:
	\begin{enumerate}
		\item if $Y\nn X$ and $X\in \ca C,$ then $X/Y\in \ca C;$
		\item if $X/\frat(X)\in \ca C,$ then $X\in \ca C.$
	\end{enumerate}
	If $N$ is a normal subgroup of a finite group $G$ and $G/N\in \ca C,$ then there exists $H\in \ca C$ such that $G=HN.$
\end{prop}
\begin{proof}
	We prove the statement by induction on the order of the group. If $N\leq \frat(G),$ then $G/\frat(G)$ is an epimorphic image of $G/N$, so
	by (1) $G/\frat(G) \in \ca C$ and therefore it follows from (2)
	that $G\in \ca C.$
	So we may assume $N\not\leq \frat(G).$ In this case, there exists a maximal subgroup $M$ of $G$ such that $G=MN.$ Moreover, $M/M\cap N \cong MN/\cap N= G/N \in  \ca C$ so there exists $H\in \ca C$ such that
	$M=(M\cap N)H$ and therefore $G=NM=N(M\cap N)H=NH.$
\end{proof}

The previous statement does not in general hold if we do not assume that $\ca C$ satisfies (2). For example if $\ca C$ is the class of the finite abelian groups and $G$ is the quaternion group of order $8$, then $G/\frat(G)\in \ca C$ but $\frat(G)$ does not admit an abelian supplement. However, Proposition \ref{class} holds even when 
$\ca C$ is the class of the finite abelian groups in the particular case when $G$ is a finite almost simple group and $N$ is the socle of $G$. The main result of this paper in fact is the following.

\begin{thm}\label{almost}
	Let $G$ a finite almost simple group with socle $G_0$. If $G/G_0$ is abelian, then $G$ contains an abelian subgroup $A$ such that $G=AG_0.$
\end{thm}

		The proof of Theorem \ref{almost} is articulated in various cases which are proved separately along the paper. Table \ref{table_proof} contains, for every non-abelian simple group $G_0$, the location of the corresponding proof.
		
		\

	\begin{table}[h]
	\caption{\label{table_proof}The proof of Theorem \ref{almost} in the various cases. Notice that ${\rm Alt}_6\cong\psl{2}{9}$ has been considered in the linear one.}
	\begin{tabular}{ccc}
		\toprule
		&$G_0$ & Reference\\
		\toprule
		alternating & ${\rm Alt}_n$, $G_0\neq{\rm Alt}_6$ & Corollary \ref{pr_alt_spor}
		\\
		\midrule
		\multirow{5}*{classical}&	$A_{n-1}(q)=\psl{n}{q}$ & Theorem \ref{pr_psl}\\
		&${}^2A_{n-1}(q)=\psu{n}{q}$ & Theorem \ref{pr_psu}\\
		&$B_n(q),C_n(q)$ & Theorem \ref{CnBnE7}\\
		&$D_n(q)$ & Theorem \ref{pr_Dnodd}, \ref{pr_Dneven}\\
		&${}^2D_n(q)$ & Theorem \ref{Dntw}\\
		\midrule
		\multirow{6}*{exceptional}&$E_6(q)$ & Theorem \ref{E6}\\
		&${}^2E_6(q)$ & Theorem \ref{E6tw}\\
		&$E_7(q)$ & Theorem \ref{CnBnE7}\vspace{.3em}\\
		&${{}^3\!D_4(q)},E_8(q),F_4(q),G_2(q)$,	&\multirow{2}*{Corollary \ref{pr_alt_spor}}\\	&${}^2B_2(2^{r}),{}^2G_2(3^{r}),{}^2F_4(2^{r})'$&\vspace{.3em}\\
		\midrule
		sporadic &all& Corollary \ref{pr_alt_spor}\\
		\bottomrule
	\end{tabular}
\end{table}

\

This result has also some consequences beyond almost simple groups, in fact we will prove the following corollary as well, on groups with $\fit(G)=1$, where $\fit(G)$ is the Fitting subgroup of $G$.

\begin{cor}\label{fittrivial}
	Let $G$ be a finite group and suppose that $\fit(G)=1$. Let  $N=\soc(G).$  If $a, b$ are two elements of $G$ and $[a,b]\in N,$ then there exist $n, m \in N$ such that $[an,bm]=1.$
\end{cor}

We now describe an application of the previous corollary, that was our original motivation to look for results in this direction. Let $G$ be a finite non-cyclic group and denote by $d(G)$ the smallest cardinality of a generating set of $G.$ The \emph{rank graph} $\Gamma(G)$ associated to $G$ is the graph whose vertices are the elements of $G$ and where $x$ and $y$ are adjacent vertices if there exists a generating set $X$ of $G$ of cardinality $d(G)$ such that $\{x,y\}$ is a subset of $X.$ When $d(G)=2,$ the graph $\Gamma(G)$ is known with the name of \emph{generating graph} of $G$ and it has been widely studied by several authors, as survey references we recommend \cite{survey_Tim} and \cite{survey_Scott}. A known open question about this graph is whether the subgraph of the generating graph of $G$ induced by its non-isolated vertices is connected. This question is considered quite difficult and related to properties of almost simple groups. It is known that the answer is affirmative if $G$ is soluble \cite{CL3} or if $G$ is a group whose proper quotients are all cyclic, in particular simple groups \cite{bgh}, but only partial results are known for arbitrary finite groups. Clearly the same question can be asked in the more general case when $\Delta(G)$ is the  subgraph of the rank graph $\Gamma(G)$ induced by its non-isolated vertices. In a paper in preparation, Corollary \ref{fittrivial} is used to prove the following result.

\begin{thm}\label{thm_conn}
	If $d(G)\geq 3,$ then $\Delta(G)$ is connected.
\end{thm}

When $d(G)=2$, the techniques used to prove Theorem \ref{thm_conn} encounter some obstacles, but they can suggest a starting point for the case of the generating graph as well.

\

We conclude this introduction by giving an outline of the structure of the paper. We begin with Section \ref{sec_pre} in which we set the stage with some notation and preliminary results. Then, in Sections \ref{sec_lin} and \ref{sec_uni} we deal respectively with linear and unitary groups. After that, in Section \ref{sec_lie} we give more details on Chevalley groups which will be the framework in which we deal with the remaining cases.
\begin{itemize}
	\item Section \ref{sec_BCE7}: groups of type $C_n(q)$, $B_n(q)$, $n\geq 2$ and $E_7(q)$;
	\item Sections \ref{sec_E6} and \ref{sec_2E6}: groups of type $E_6(q)$ and ${}^2\!E_6(q)$;
	\item Section \ref{sec_2D}: groups of type ${}^2\!D_n(q)$;
	\item Sections \ref{sec_Dodd} and \ref{sec_Deven}: groups of type $D_n(q)$.
\end{itemize}

Finally, in Section \ref{sec_corollary} we conclude with the proof of Corollary \ref{fittrivial}.

\section{Notation and preliminary results}\label{sec_pre}

In this section we will present the main strategy for the proof of Theorem \ref{almost} and prove some preliminary results which will also  establish the main theorem for some families of almost simple groups. 
 
 \
 
Let $G_0$ be a finite non-abelian simple group and let
\[\rho:\aut(G_0)\to\out(G_0)\cong \aut(G_0)/G_0\] be the canonical projection. The following definition will provide the language we will use in the proof of our main result.

\begin{defn}\label{pr}
	Let $T$ be an abelian subgroup of $\out(G_0)$. We say that $\tilde{T}\leq\aut(G_0)$ is a $T$-abelian supplement if
	\begin{enumerate}
		\item $\tilde T$ is abelian.
		\item $\rho(\tilde{T})=T$.
	\end{enumerate}

\end{defn}

Notice in particular that if $T$ and $\tilde T$ are as in the previous definition, if $G/G_0\cong T$, then $G=\tilde TG_0$ with $\tilde T$ abelian, therefore
 proving Theorem \ref{almost} is equivalent to proving  that for every non-abelian simple group $G_0$ and every abelian $T\leq\out(G_0)$, there exists a $T$-abelian supplement. The strategy of the proof of Theorem \ref{almost} is in fact the following: given $G_0$, we analyse all the abelian subgroups $T$ of $\out(G_0)$ and prove case by case that there exists a $T$-abelian supplement. Actually, it is not necessary to check each abelian subgroup of $\out(G_0)$, but only the maximal abelian ones, as it is shown by the following lemma.
 
 \begin{lemma}\label{lemma_max}
 	Let $T\leq S\leq\out(G_0)$ with $T$ and $S$ abelian. If there exists a $S$-abelian supplement, then there exists a $T$-abelian supplement as well.
 \end{lemma}
\begin{proof}
	Let $\tilde S$  be an $S$-abelian supplement. Let $\tilde T$ be the preimage of $T$ by the map $\rho\vert_{\tilde S}$. Then $\tilde T\leq\tilde S$ and so it is abelian, moreover $\rho(\tilde T)=\rho\vert_{\tilde S}(\tilde T)=T$ and so $\tilde T$ is a $T$-abelian supplement.
\end{proof}

In particular, whenever $\out(G_0)$ is abelian, to prove Theorem \ref{almost},  it is enough to check that there exists a $\out(G_0)$-abelian supplement.

We will now prove a couple of important lemmas which guarantee the existence of $T$-abelian supplements in some circumstances. 

\begin{lemma}\label{pr_cyclic}
	Let $T$ be a cyclic subgroup of $\out(G_0)$. Then there exists a $T$-abelian supplement.
\end{lemma}
\begin{proof}
	Let $T=\gen{t}$ and let $\tilde t\in\aut(G_0)$ be a preimage of $t$ under $\rho$. Then $\tilde T=\langle \tilde t \rangle$ is a $T$-abelian supplement.
\end{proof}

The previous lemma, together with Lemma \ref{lemma_max}, shows that

\begin{cor}\label{pr_alt_spor}
	Theorem \ref{almost} is valid for all the almost simple groups with socle $G_0$ such that $\out(G_0)$ is cyclic.
	
	In particular, our main result is established whenever 
	\begin{itemize}
		\item $G_0={\rm Alt}_n$, $n\geq5$, with $n\neq 6$;
		\item $G_0={}^3\!D_4(q),E_8(q),F_4(q), G_2(q), \,{}^2B_2(2^{r}),{}^2G_2(3^{r}),{}^2F_4(2^{r})'$;
		\item $G_0$ is a sporadic simple group.
	\end{itemize}
\end{cor}

 Noticing that ${\rm Alt}_6\cong\psl{2}{9}$, this corollary reduces our investigation to the groups of Lie type only.

\

In what follows, we denote by $\mathbb{F}_q$ the field with $q=p^m$ elements, where $p$ is a prime. Moreover, we denote with $\omega$ a generator of $\mathbb{F}_q^\times$. 
Let $G_0$ be a simple group of Lie type over $\mathbb{F}_q$. We denote by $d$ the index of $G_0$ in ${\rm Inndiag}(G_0)$, the subgroup of $\aut(G_0)$ generated by the inner and diagonal automorphisms of $G_0$. We give the values of $d$ in Table~\ref{table_d}, to provide a quick reference to look up, since such values play a central role in the proofs.

\begin{table}[h]
	\caption{\label{table_d}The values of $d$ for simple groups of Lie type.}
	\begin{tabular}{ccc}
		\toprule
		&$G_0$ & $d$\\
		\toprule
		\multirow{6}*{untwisted}&	$A_{n-1}(q)=\psl{n}{q}$ & $(n,q-1)$\\
		&$B_n(q),C_n(q)$ & $(q-1,2)$\\
		&$D_n(q)$ & $(4,q^n-1)$\\
		&$E_6(q)$ & $(3,q-1)$\\
		&$E_7(q)$ & $(2,q-1)$\\
		&$E_8(q),F_4(q),G_2(q)$&$1$ \\
		\midrule
		\multirow{4}*{twisted}
		&${}^2A_{n-1}(q)=\psu{n}{q}$ & $(n,q+1)$\\
		&${}^2D_n(q)$ & $(4,q^n+1)$\\
		&${}^2E_6(q)$ & $(3,q+1)$\\
		&${}^2B_2(2^{r}),{}^3\!D_4(q),{}^2G_2(3^{r}),{}^2F_4(2^{r})$&$1$\\
		\bottomrule
	\end{tabular}
\end{table}
 
We are now able to state another fundamental ingredient for the proof of Theorem~\ref{almost}.

\begin{lemma}\label{pr_complement}
	If $\aut(G_0)$ splits over $G_0$, then there exists a $T$-abelian supplement for every abelian $T\leq\out(G_0)$.
\end{lemma}
\begin{proof}
	Let $H$ be a complement of $G_0$ in $\aut(G_0)$. Then $H\cong\out(G_0)$ and the subgroup of $H$ corresponding to $T$ is a $T$-abelian supplement.
\end{proof}

In \cite{comp}, A. Lucchini, F. Menegazzo and M. Morigi gave a complete classification of all simple groups of Lie type $G_0$ for which $\aut(G_0)$ splits over $G_0$. Their main result is the following.

\begin{thm}\label{th_comp}
	Let $G_0$ be a simple group of Lie type, $q=p^m$.
	 Then $\aut(G_0)$ splits over $G_0$ if and only if one of the following
	conditions holds:
	\begin{enumerate}
		\item $G_0$ is untwisted, not of type $D_n(q)$, and $( \frac{q-1}{d}
		, d, m) = 1$;
		\item  $G_0 = D_n(q)$ and $( \frac{q^n-1}{d} , d, m) = 1$;
		\item  $G_0$ is twisted, not of type ${}^2D_n(q)$, and $(\frac{q+1}{d} , d, m) = 1$;
		\item $G_0 = {}^2D_n(q)$ and either $n$ is odd or $p = 2$.
	\end{enumerate}
\end{thm}


We are now ready to begin the investigation of the various types of almost simple groups, starting with the ones with linear socle.

\section{Linear groups}\label{sec_lin}

In this section we prove Theorem \ref{almost} in the linear case. We begin with the easiest case $n=2$, which is better understood on its own and gives us an explicit model for the more general setting. Then we deal with the case $n\geq3$. More specifically, we prove some technical lemmas and analyse all the different types of abelian subgroups $T$ of the outer automorpfism group, showing the existence of $T$-abelian supplements in each case. Finally, the main result of this section is contained in Theorem \ref{pr_psl}.

\begin{thm}\label{pr_psl2}
	Let $G$ a finite almost simple group with socle $G_0=\psl{2}{q}$. Then $G$ contains an abelian subgroup $A$ such that $G=AG_0.$
\end{thm}

\begin{proof}

Let $Z:=\Z(\gl{2}{q})$. We can suppose that $q$ is odd, otherwise $d=1$ and $\aut(G_0)$ splits over $G_0$. The outer automorphism group in this case is the following:
\[
\out(G_0)=\gen{\delta}\times\gen{\phi},
\]
where $\delta$ is the diagonal automorphism with $|\delta|=2$ and $\phi$ is the field automorphism with $|\phi|=m$ and $[\delta,\phi]=1$.

Let
\[
A:=\begin{pmatrix}
0 & -\omega\\
1 & 0
\end{pmatrix}.\qquad
B:=\begin{pmatrix}
\omega^{\frac{p-1}{2}} & 0\\
0 & 1
\end{pmatrix}.
\]

We have 
\[
A^{\phi B}=\begin{pmatrix}
\omega^{\frac{1-p}{2}} & 0\\
0 & 1
\end{pmatrix}\begin{pmatrix}
0 & -\omega^p\\
1 & 0
\end{pmatrix}\begin{pmatrix}
\omega^{\frac{p-1}{2}} & 0\\
0 & 1
\end{pmatrix}=\begin{pmatrix}
0 & -\omega^{\frac{p+1}{2}}\\
\omega^{\frac{p-1}{2}} & 0
\end{pmatrix}=\omega^{\frac{p-1}{2}}A.
\]
Therefore
\[
[A,\phi B]\in Z
\]
and $\rho(AZ)=\delta$ and $\rho(\phi BZ)$ can be $\phi\delta$ or $\phi$, but in any case
\[
\rho(\gen{AZ,\phi BZ})=\out(G_0)
\]
and therefore
\[
\gen{A,\phi B}Z/Z
\]
is a $\out(G_0)$-abelian supplement.
\end{proof}

\

From now on, $G_0:=\psl{n}{q}$ with $n\geq 3$, so $d=(n,q-1)$. We write also  $Z:=\Z(\gl{n}{q})$. 

In this case, the outer automorphism group is the following:
\[
\out(G_0)=\gen{\delta}\sd\gen{\phi,\gamma},
\]
where $\delta$ is a diagonal automorphism with $|\delta|=d$, $\phi$ is the field automorphism which raises the coefficients of every matrix to the power of $p$ and $\gamma$ is the graph automorphism which transform each matrix in its inverse transpose. In particular we have $|\phi|=m$, $|\gamma|=2$,
$[\phi,\gamma]=1$, $\delta^\phi=\delta^p$ and 
$\delta^\gamma=\delta^{-1}$.

\

In the sequel we will use a lot the following special matrices defined from some integers $w,l,c	\in\mathbb{Z}$ with $w\geq2$:

\[
A_{w,l}:=\begin{pmatrix}
0 &0 &\dots & 0&(-1)^{w-1}\omega^{l}\\
1 & 0  & \dots &0& 0\\
0 & 1  & \dots &0& 0\\
\vdots& \vdots &\ddots &\vdots &\vdots\\
0 & 0  & \dots&1 & 0
\end{pmatrix}\in\gl{w}{q}
\]
and
\[
X_{w,c}:=\begin{pmatrix}
\omega^{c(w-1)} &0 &\cdots & 0&0\\
0 & \omega^{c(w-2)}  & \cdots &0& 0\\
\vdots & \vdots  & \ddots &\vdots& \vdots\\
0& 0 &\cdots &\omega^{c} &0\\
0 & 0  & \cdots&0 & 1
\end{pmatrix}\in\gl{w}{q}.
\]

Notice that

\[
\det{A_{w,l}}=\omega^l.
\]

We now introduce a technical lemma which is the key ingredient of the proofs in this section.

\begin{lemma}\label{block}
	Let $w, l, c\in\mathbb{Z}$ be integers. Let \[
	A:=A_{w,l}\qquad X:=X_{w,c}.
	\]
	 If
\[cw\equiv lp^s(-1)^\varepsilon -l \mod q-1,\]
	then we have
	\[
	A^{\phi^s\gamma^\varepsilon X}=\omega^{c}A.
	\]
\end{lemma}

\begin{proof}
	First notice that
	\[
	A^{\phi^s\gamma^\varepsilon}=\begin{pmatrix}
	0 &0 &\dots & 0&(-1)^{w-1}\omega^{lp^s(-1)^\varepsilon}\\
	1 & 0  & \dots &0& 0\\
	0 & 1  & \dots &0& 0\\
	\vdots& \vdots &\ddots &\vdots &\vdots\\
	0 & 0  & \dots&1 & 0
	\end{pmatrix}.
	\]
	To prove the equality, we will check that $A^{\phi^s\gamma^\varepsilon X}$ and $\omega^{c}A$ act in the same way on the canonical basis' vectors $\{e_i\mid  1\leq i\leq w\}$.

	We have that \[Xe_i=\omega^{c(w-i)}e_i, \qquad1\leq i\leq w.\]
	If $i<w$, then \[A^{\phi^s\gamma^\varepsilon}e_i=e_{i+1}=Ae_i,\] therefore
	\[
	A^{\phi^s\gamma^\varepsilon X}e_i=X^{-1}A^{\phi^s\gamma^\varepsilon }Xe_i=\omega^{c(w-i)}X^{-1}A^{\phi^s\gamma^\varepsilon }e_i=	\]
	\[=\omega^{c(w-i)}X^{-1}e_{i+1}=\omega^{c(w-i)}\omega^{-c(w-i-1)}e_{i+1}=\omega^{c}e_{i+1}=\omega^{c}Ae_i.
	\]
	On the other hand, if $i=w$ we have
	\[
	A^{\phi^s\gamma^\varepsilon X}e_w=X^{-1}A^{\phi^s\gamma^\varepsilon }Xe_w=X^{-1}A^{\phi^s\gamma^\varepsilon }e_w=
	\]
	\[
	(-1)^{w-1}\omega^{lp^s(-1)^\varepsilon}X^{-1}e_1=(-1)^{w-1}\omega^{lp^s(-1)^\varepsilon}\omega^{-c(w-1)}e_1=
	\]
	\[
	(-1)^{w-1}\omega^{lp^s(-1)^\varepsilon-cw+c}e_1=(-1)^{w-1}\omega^{l+c}e_1,
	\]
	since \[cw\equiv lp^s(-1)^\varepsilon -l \mod q-1.\] Finally,
	\[
	\omega^{c}Ae_w=\omega^{c}\cdot (-1)^{w-1}\omega^{l}e_1=(-1)^{w-1}\omega^{l+c}e_1,
	\]
	so the two linear maps coincide on a basis and therefore they are equal in $\gl{w}{q}$.
\end{proof}

We are now able to present the existence of $T$-abelian supplements for some of the possible choices of $T$.

\begin{prop}\label{pgammal}
	Let 
	$T=\gen{\delta^k,\phi^s\gamma^\varepsilon\delta^j}$
	 with $k\mid d$ and $k\not=d$. Then there exists a $T$-abelian supplement.
\end{prop}

\begin{proof}
	  Since $T$ is abelian, \[\delta^k=(\delta^k)^{\phi^s\gamma^\varepsilon\delta^j}=\delta^{kp^s(-1)^\varepsilon},\] which means $d\mid k((-1)^\varepsilon p^s-1)$ or, equivalently  \[t\mid 
	  (-1)^\varepsilon p^s-1, \qquad t:=d/k.\]
	  
	  First, suppose $t=n$, so $t=d=n$ and $k=1$. In this case $T=\gen{\delta,\phi^s\gamma^\varepsilon}$ and 
	  \[
	  \tilde T:=\gen{A_{n,1},\phi^s\gamma^\varepsilon X_{n,\frac{(-1)^\varepsilon p^s-1}{n}}}Z/Z
	  \]
	  is a $T$-abelian supplement, since $\rho(A_{n,1}Z)=\delta$ and by applying Lemma \ref{block} with $$(w,l,c)=\left(n,1,\frac{(-1)^\varepsilon p^s-1}{n}\right),$$ we have 
	  \[
	  \left[A_{n,1},\phi^s\gamma^\varepsilon X_{n,\frac{(-1)^\varepsilon p^s-1}{n}}\right]\in Z.
	  \]
	  
	  So in the sequel we can suppose $t\neq n$.
	\begin{step}
		We find an integer $y\in\mathbb{Z}$ \st $yn\equiv d \mod{q-1}$ and $(y,t)=1$.
	\end{step}
	Since $\left(\frac{n}{d},\frac{q-1}{d}\right)=1$, there exists $\ov{y}\in\mathbb{Z}$ \st $\ov{y}\frac{n}{d}\equiv 1 \mod \frac{q-1}{d}$. Now let \[
	t=p_1^{\al_1}\dots p_{\tilde l}^{\al_{\tilde l}}p_{\tilde l+1}^{\al_{\tilde l+1}}\dots p_l^{\al_l}
	\] its prime factorization, where we ordered the primes in a way \st $p_i$ divides $\ov{y}$ \ifa $1\leq i\leq \tilde l$.
	
	Let 
	\[
	y=\ov{y}+p_{\tilde l+1}\cdots p_l\frac{q-1}{d}.
	\]
	
	For every $p_i$, we have that $p_i$ does not divide $y$, because if $1\leq i\leq \tilde l$, then $\ov{y}$ is divisible by $p_i$ while $p_{\tilde l+1}\cdots p_l\frac{q-1}{d}$ is not (since $(\ov{y},\frac{q-1}{d})=1$) and if $\tilde l<i\leq l$, $p_i$ divides $p_{\tilde l+1}\cdots p_l\frac{q-1}{d}$ but not $\ov{y}$. Therefore $(y,t)=1$, moreover $y\equiv\ov{y}\mod \frac{q-1}{d}$ and $yn\equiv d\mod{q-1}$.
	
	\begin{step}\label{1:st2}
		We construct matrices $A,X\in\gl{n}{q}$ \st $\det{A}=\omega^k$ and $[A,\phi^s\gamma^\varepsilon X]\in Z$. 
	\end{step}
	Since $t\mid n$ and $t\neq1,n$, we have that both $t\geq2$ and $n-t\geq2$  and so we can define
	\[
	A:=\begin{pmatrix}
	A_{t,y} & 0\\
	0 & A_{n-t,k-y}
	\end{pmatrix}.
	\]
	
	First, notice that 
	\[
	\det{A}=\det{A_{t,y}}\det{A_{n-t,k-y}}=\omega^y\omega^{k-y}=\omega^k,
	\]
	therefore $\rho(AZ)=\delta^k$.
	
	Let $r:=((-1)^\varepsilon p^s-1)/t$ and define
	
	 \[
	 X:=\begin{pmatrix}
	 X_{t,yr} & 0\\
	 0 & X_{n-t,yr}
	 \end{pmatrix},
	 \]
	
	We have that 
	\[
	y(p^s(-1)^\varepsilon-1)=yrt \mod{q-1},
	\]
	
	so applying Lemma \ref{block} with $(w,l,c)=(t,y,yr)$ we get
	
	\[
	A_{t,y}^{\phi^s\gamma^\varepsilon X_{t,yr}}=\omega^{yr}A_{t,y}.
	\]
	
	Moreover, recalling that $kt=d\equiv ny \mod{q-1}$, we have that 
	\[
	(k-y)(p^s(-1)^\varepsilon-1)=(k-y)rt=krt-yrt=yr(n-t) \mod{q-1},
	\]
	
	so applying Lemma \ref{block} with $(w,l,c)=(n-t,k-y,yr)$ we get
	
	\[
	A_{n-t,k-y}^{\phi^s\gamma^\varepsilon X_{n-t,yr}}=\omega^{yr}A_{n-t,k-y}.
	\]

	Therefore we have
	\[
	A^{\phi^s\gamma^\varepsilon X}=\omega^{yr}A
	\]
	or, equivalently,
	\[
	[A,\phi^s\gamma^\varepsilon X]\in Z.
	\]
	
	\begin{step}\label{1:st3}
		We find a matrix  $C\in\gl{n}{q}$ \st $[A,C]=1$ with $\det{C}=\omega$.
	\end{step}

	Recall that $(y,t)=1$, so there exist $a,b\in\mathbb{Z}$ \st $ay+bt=1$.
	Let 
	\[
	C_0:=\omega^b A_{t,y}^a \in\gl{t}{q}.
	\]
	We have that $[A_{t,y},C_0]=1$ and 
	\[
	\det{C_0}=\det{A_{t,y}}^a\omega^{bt}=\omega^{ay+bt}=\omega.
	\]
	Let 
	\[
	C:=\begin{pmatrix}
	C_0 & 0\\
	0 & 1
	\end{pmatrix}.
	\]
	We have $[A,C]=1$ and $\det{C}=\det{C_0}=\omega$.
	\begin{step}
		We complete the proof by constructing a $T$-abelian supplement.
	\end{step}

	Let $u\in\mathbb{Z}$ \st $\rho(XZ)=\delta^u$. 
	Combining Steps \ref{1:st2} and \ref{1:st3}, we get
	\[
	[A,\phi^s\gamma^\varepsilon XC^{j-u}]\in Z,
	\]
	with $\rho(AZ)=\delta^k$ and $\rho(XC^{j-u}Z)=\delta^u\delta^{j-u}=\delta^j $.
	Therefore
	\[
	\tilde T:=\gen{A,\phi^s\gamma^\varepsilon  XC^{j-u}}Z/Z
	\]
	is a $T$-abelian supplement.
	
\end{proof}

To continue our investigation, we need another couple of small lemmas.

\begin{lemma}\label{comm}
	Let $A,B\in\gl{w}{q}$. Then
	\[
	[\phi^sA,\gamma B]=1
	\]
	\ifa
	\[
	B=A^TB^{\phi^s}A.
	\]
\end{lemma}
\begin{proof}
	Easy computation.
\end{proof}

\begin{lemma}\label{gX}
	If $\alpha,\beta\in\mathbb{Z}$ are \st 
	\[
	\beta\equiv 2\alpha+p^s\beta \mod{q-1},
	\]
	then
	\[
	[\phi^sX_{w,\alpha},\gamma X_{w,\beta}]=1.
	\]
\end{lemma}
\begin{proof}
	Since $X_{w,\alpha},X_{w,\beta}$  are diagonal, for Lemma \ref{comm} we just need to check that
	\[X_{w,\beta}=X_{w,\alpha}X_{w,\beta}^{p^s}X_{w,\alpha}.\]
	
	By inspecting the coefficients on the diagonal, for every $1\leq i\leq w$ we have 
	\[
	\omega^{\beta(w-i)}=\omega^{\alpha(w-i)}\omega^{p^s\beta(w-i)}\omega^{\alpha(w-i)}=\omega^{(2\alpha+p^s\beta)(w-i)},
	\]
	which is true because of the hypothesis on $\alpha$ and $\beta$. 
\end{proof}

We now show the existence of $T$-abelian supplements for other choices of $T$.

\begin{prop}\label{3gen}
	Let $d$ be even, then we can find a $T$-abelian supplement for $T$ of the form $T=\gen{\delta^{d/2},\phi^s\delta^j,\gamma\delta^k}$. 
\end{prop}

\begin{proof} As in the previous case, this proof is articulated in different steps.
	\setcounter{step}{0}
	\begin{step}
		We find an integer $y\in\mathbb{Z}$ \st $yn\equiv d \mod{q-1}$ and $y$ is odd.
	\end{step}

	Since $\left(\frac{n}{d},\frac{q-1}{d}\right)=1$, there exists $\ov{y}\in\mathbb{Z}$ \st $\ov{y}\frac{n}{d}\equiv 1 \mod \frac{q-1}{d}$. If $\ov{y}$ is odd,  define $y:=\ov{y}$; if $\ov{y}$ is even, therefore $\frac{q-1}{d}$ is odd, define $y:=\ov{y}+\frac{q-1}{d}$ which is odd and \st $y\frac{n}{d}\equiv 1 \mod \frac{q-1}{d}$. So, we have $yn\equiv d \mod{q-1}$ with $y$ odd.
	
	\begin{step}
		We construct matrices $A,X_\phi,X_\gamma\in\gl{n}{q}$ \st $\det{A}=\omega^{d/2}$ and $\hat{T_1}:=\gen{A,\phi^sX_\phi,\gamma X_\gamma}Z/Z$ is abelian.
	\end{step}

	Since $n-2\geq2$ we can define 
	
	\[
	A:=\begin{pmatrix}
	A_{2,y} & 0\\
	0 & A_{n-2,d/2-y}
	\end{pmatrix},
	\]
	so that $\rho(AZ)=\delta^{d/2}$ since
	\[
	\det{A}=\det{A_{2,y}}\det{A_{n-2,d/2-y}}=\omega^y\omega^{d/2-y}=\omega^{d/2}.
	\]
	
	Let $r:=(p^s-1)/2$. 
	Considering the automrphisms $\phi^s$ and $\gamma$ let us now argue as in Step \ref{1:st2} of Propostion \ref{pgammal} and construct 
	\[
	X_\phi:=\begin{pmatrix}
	X_{2,yr} & 0\\
	0 & X_{n-2,yr}
	\end{pmatrix},
	\]
	
	so that
	\[
	A^{\phi^sX_\phi}=\omega^{yr}A
	\]
	and
	\[
	X_\gamma:=\begin{pmatrix}
		X_{2,-y} & 0\\
		0 & X_{n-2,-y}
	\end{pmatrix}
	\]
	
	so that
	
	\[
	A^{\gamma X_\gamma}=\omega^{-y}A.
	\]
	
	Since 
	\[
	-y\equiv 2yr-yp^s \mod{q-1},
	\]
	for Lemma \ref{gX} we have \[
	\left[\phi^sX_{2,yr},\gamma X_{2,-y}\right]=\left[\phi^sX_{n-2,yr},\gamma X_{n-2,-y}\right]=1\]
	and therefore
	\[
	\left[\phi^sX_\phi,\gamma X_\gamma\right]=1.
	\]
	
	From this, we obtain that
	\[
	\hat{T_1}:=\gen{A,\phi^sX_\phi,\gamma X_\gamma}Z/Z
	\]
	is abelian.

	\begin{step}
	We construct matrices $X_\phi',X_\gamma'\in\gl{n}{q}$ \st $\det{X_\gamma'}=\omega^y\det{X_\gamma}$ and  $\hat{T_2}:=\gen{A,\phi^sX_\phi',\gamma X_\gamma'}Z/Z$ is abelian. 
	\end{step}
	
	Let us define
	
	\[
	Y_\phi:=X_{2,yr}A_{2,y}^{-r},\quad Y_\gamma:=X_{2,-y}A_{2,y}=\begin{pmatrix}
	0 & -1\\
	1 & 0
	\end{pmatrix},\]
	
	\[
	C_\phi:=\begin{pmatrix}
	A_{2,y}^{-r} & 0\\
	0 & 1
	\end{pmatrix}, \quad 
	C_\gamma:=\begin{pmatrix}
	A_{2,y} & 0\\
	0 & 1
	\end{pmatrix}
	\]
	and
	\[
	 X_\phi':=X_\phi C_\phi=\begin{pmatrix}
	Y_\phi & 0\\
	0 & X_{n-2,yr}
	\end{pmatrix}, \quad X_\gamma':=X_\gamma C_\gamma=\begin{pmatrix}
	Y_\gamma & 0\\
	0 & X_{n-2,-y}
	\end{pmatrix}.
	\]
	
	Since $C_\phi,C_\gamma\in\C_{\gl{n}{q}}(A)$, we have
	\[
		A^{\phi^sX_\phi'}=\omega^{yr}A\qquad	A^{\gamma X_\gamma'}=\omega^{-y}A.
	\]
	
	Notice that 	
	
	\[
	X_{2,yr}^TY_\gamma X_{2,yr}=\begin{pmatrix}
	\omega^{yr} & 0\\
	0 & 1
	\end{pmatrix}\begin{pmatrix}
	0 & -1\\
	1 & 0
	\end{pmatrix}\begin{pmatrix}
	\omega^{yr} & 0\\
	0 & 1
	\end{pmatrix}=\omega^{yr}Y_\gamma
	\]
	
	and
	
		\[
	A_{2,y}^TY_\gamma A_{2,y}=\begin{pmatrix}
	0 & 1\\
	-\omega^y & 0
	\end{pmatrix}\begin{pmatrix}
	0 & -1\\
	1 & 0
	\end{pmatrix}\begin{pmatrix}
	0 & -\omega^y\\
	1 & 0
	\end{pmatrix}=\begin{pmatrix}
	0 & -\omega^y\\
	\omega^y & 0
	\end{pmatrix}=\omega^yY_\gamma.
	\]
	
	From this we deduce
	\[
	Y_\phi^TY_\gamma^{\phi^s}Y_\phi=\left(A_{2,y}^T\right)^{-r} X_{2,yr}^T Y_\gamma  X_{2,yr}  A_{2,y}^{-r}=
	\]
	\[
	=\omega^{yr}\left(A_{2,y}^T\right)^{-r} Y_\gamma  A_{2,y}^{-r}=Y_\gamma.
	\]
	Since $X_{n-2,yr}^TY_{n-2,-y}^{\phi^s}X_{n-2,yr}=Y_{n-2,-y}$, this means that
	\[
	\left[\phi^sX_\phi',\gamma X_\gamma'\right]=1.
	\]
	
	Therefore
		\[
	\det{X_\gamma'}=\det{X_\gamma}\det{C_\gamma}=\det{X_\gamma}\det{A_{2,y}}=\omega^y\det{X_\gamma}.
	\]
	and
	\[
	\hat{T_2}:=\gen{A,\phi^sX_\phi',\gamma X_\gamma'}Z/Z
	\]
	is abelian.
	
	\begin{step}
		We complete the proof by constructing a $T$-abelian supplement.
	\end{step}
	 
	 Let $\rho(X_\gamma Z)=\delta^u$ for some $u\in\mathbb{Z}$, so $\rho(X_\gamma' Z)=\delta^{y+u}$. Given that $y$ is odd, one of $u-k$ or $y+u-k$ is even. Since $\gamma\delta^{x}$ is conjugate to $\gamma\delta^{y}$ in $\gen{\delta,\gamma}$ if $y-x$ is even, one of $\gamma\delta^{u}$ or $\gamma\delta^{u+y}$ is conjugate to $\gamma\delta^k$. 
	
	Let 
	\[
	\hat{T}:=\begin{cases}
	\hat{T_1}\qquad \text{if $u-k$ is even}\\
	\hat{T_2}\qquad \text{if $u+y-k$ is even}
	\end{cases},
	\]
so that there exists a matrix $R\in\gl{n}{q}$ \st $\rho(\hat{T}^R)=\gen{\delta^{d/2},\phi^s\delta^l,\gamma\delta^k}$ for some $l\in\mathbb{Z}$. Notice that this group being abelian means $2l\equiv-k(p^s-1) \mod{d}$. In the same way, since $T=\gen{\delta^{d/2},\phi^s\delta^j,\gamma\delta^k}$ is abelian, it means that $2j\equiv-k(p^s-1) \mod{d}$, but then $2l\equiv 2j \mod{d}$, which means $l\equiv j \mod{d/2}$ and $\rho(\hat{T}^R)=T$. Therefore $\tilde{T}:=\hat{T}^R$ is a $T$-abelian supplement.
\end{proof}

We are now able to prove Theorem \ref{almost} in the linear case.

\begin{thm}\label{pr_psl}
Let $G$ a finite almost simple group with socle $G_0=\psl{n}{q}$. If $G/G_0$ is abelian, then $G$ contains an abelian subgroup $A$ such that $G=AG_0.$
\end{thm}

\begin{proof} If $n=2$, this follows from Theorem \ref{pr_psl2}, so we can suppose $n\geq3$. The statement in this case is equivalent to finding a $T$-abelian supplement for every abelian $T\leq\out(G_0)$.
	Let $\pi\colon\out(G_0)\to\out(G_0)/\gen{\delta}=\gen{\phi,\gamma}$. If $\pi(T)=\gen{\phi^s\gamma^\varepsilon}$ is cyclic, then $T$ is of the form $T=\gen{\delta^k,\phi^s\gamma^\varepsilon\delta^j}$ with $k\mid d$ and we conclude by Proposition \ref{pgammal}. If $\pi(T)$ is not cyclic, then $\pi(T)=\gen{\phi^{s},\gamma}$. Suppose $d$ is odd. Then $T$ is two generated and of the form $T=\gen{\phi^{s}\delta^j,\gamma\delta^k}$. Since $d$ is odd $\gamma\delta^k$ is conjugate to $\gamma$ in $\gen{\delta,\gamma}$, so, up to conjugation, we can assume $k=0$ and therefore $\delta^j=1$, since $[\phi^s\delta^j,\gamma]=1$; therefore $\tilde{T}:=T$ is a $T$-abelian supplement. Suppose $d$ is even. If $T$ is $2$-generated it is of the form $T=\gen{\phi^{s}\delta^j,\gamma\delta^k}$  and since $\delta^{d/2}\in\Z(\out(G_0))$, it is contained in an abelian subgroup of $\out(G_0)$ of the form $\gen{\delta^{d/2},\phi^{s}\delta^j,\gamma\delta^k}$ and we conclude by Proposition \ref{3gen}. If it is $3$-generated, it is of the form $T=\gen{\delta^{l},\phi^{s}\delta^j,\gamma\delta^k}$ and in order to be abelian we should have $[\delta^l,\gamma\delta^k]=1$, therefore $l=d/2$ and again we conclude by Proposition \ref{3gen}.    
\end{proof}

\renewcommand{\omega}{\om}
\section{Unitary groups}\label{sec_uni}

Throughout this section, $q=p^m$, $\nu$ is a primitive element of the field $\mathbb{F}_{q^2}$ and $\omega:=\nu^{q-1}$ so that $|\omega|=q+1$. Moreover, $G_0:=\psu{n}{q}$, and so $d=(n,q+1)$. Finally, we write $Z:=\Z(\gu{n}{q})$. 

We have
\[
\out(G_0)=\gen{\delta}\sd\gen{\phi},
\]
where $\delta$ is a diagonal automorphism with $|\delta|=d$ and $\phi$ is the field automorphism which raises the coefficients of every matrix to the power $p$. In particular we have $|\phi|=2m$ and $\delta^\phi=\delta^p$.

\

As in the linear case, to prove Theorem \ref{almost} for unitary groups, we use some special matrices defined from some integers $w,l,c	\in\mathbb{Z}$ with $w\geq2$, which play a similar role that the ones in Section \ref{sec_lin}, but with a slightly different definition, where we use $\omega$ in place of $\lambda$: 

\[
A_{w,l}:=\begin{pmatrix}
0 &0 &\dots & 0&(-1)^{w-1}\omega^{l}\\
1 & 0  & \dots &0& 0\\
0 & 1  & \dots &0& 0\\
\vdots& \vdots &\ddots &\vdots &\vdots\\
0 & 0  & \dots&1 & 0
\end{pmatrix}\in\gl{w}{q}
\]
and
\[
X_{w,c}:=\begin{pmatrix}
\omega^{c(w-1)} &0 &\cdots & 0&0\\
0 & \omega^{c(w-2)}  & \cdots &0& 0\\
\vdots & \vdots  & \ddots &\vdots& \vdots\\
0& 0 &\cdots &\omega^{c} &0\\
0 & 0  & \cdots&0 & 1
\end{pmatrix}\in\gl{w}{q}.
\]

Notice that, as in the previous case,

\[
\det{A_{w,l}}=\omega^l
\]

and moreover, keeping in mind that $|\omega|=q+1$, it can be easily checked that $A_{w,l},X_{w,c}\in\gu{w}{q}$ are unitary matrices, as well as the scalar matrix $\omega$ in every dimension. Before proving the main result of this section, we present an analog of Lemma \ref{block}.

\begin{lemma}\label{ublock}
	Let $w, l, c\in\mathbb{Z}$ be integers. Let \[
	A:=A_{w,l}\qquad X:=X_{w,c}.
	\]
	If
	\[cw\equiv l(p^s-1) \mod q+1,\]
	then we have
	\[
	A^{\phi^s X}=\omega^{c}A.
	\]
\end{lemma}

\begin{proof}
	To prove the statement is sufficient to carry out the same type of computations of the proof of Lemma \ref{block}.
\end{proof}

We are now ready to prove Theorem \ref{almost} in the case in which $G_0$ is a unitary group.

\begin{prop}\label{pr_psu}
	Let $G$ a finite almost simple group with socle $G_0=\psu{n}{q}$. If $G/G_0$ is abelian, then $G$ contains an abelian subgroup $A$ such that $G=AG_0.$
\end{prop}

\begin{proof}The statement is equivalent to finding a $T$-abelian supplement for every abelian $T\leq\out(G_0)$. The abelian subgroups of $\out(G_0)$ are of the form $T=\gen{\delta^k,\phi^s\delta^j}$ with $k\mid d$, $k\neq1$, and since $T$ is abelian, $t\mid p^s-1$  with $t:=d/k$.
	
	In its structure, this proof resembles the one of Proposition \ref{pgammal} and again is articulated in different steps which focus on one such $T$ and build a $T$-abelian supplement accordingly.
	
	 First, suppose $t=n$, so $t=d=n$ and $k=1$. In this case $T=\gen{\delta,\phi^s}$ and 
	\[
	\tilde T:=\gen{A_{n,1},\phi^s X_{n,\frac{p^s-1}{n}}}Z/Z
	\]
	is a $T$-abelian supplement, since $\rho(A_{n,1}Z)=\delta$ and by applying Lemma \ref{ublock} with $$(w,l,c)=\left(n,1,\frac{ p^s-1}{n}\right),$$ we have 
	\[
	\left[A_{n,1},\phi^s X_{n,\frac{ p^s-1}{n}}\right]\in Z.
	\]
	
	So in the sequel we can suppose $t\neq n$.
	
	\setcounter{step}{0}
	\begin{step}
		We find an integer $y\in\mathbb{Z}$ \st $yn\equiv d \mod{q+1}$ and $(y,t)=1$.
	\end{step}
	Since $\left(\frac{n}{d},\frac{q+1}{d}\right)=1$, there exists $\ov{y}\in\mathbb{Z}$ \st $\ov{y}\frac{n}{d}\equiv 1 \mod \frac{q+1}{d}$. Now let \[
	t=p_1^{\al_1}\dots p_{\tilde l}^{\al_{\tilde l}}p_{\tilde l+1}^{\al_{\tilde l+1}}\dots p_l^{\al_l}
	\] its prime factorization, where we ordered the primes in a way \st $p_i$ divides $\ov{y}$ \ifa $1\leq i\leq \tilde l$.
	
	Let 
	\[
	y=\ov{y}+p_{\tilde l+1}\cdots p_l\frac{q+1}{d}.
	\]
	
	For every $p_i$, we have that $p_i$ does not divide $y$, because if $1\leq i\leq \tilde l$, then $\ov{y}$ is divisible by $p_i$ while $p_{\tilde l+1}\cdots p_l\frac{q+1}{d}$ is not (since $(\ov{y},\frac{q+1}{d})=1$) and if $\tilde l<i\leq l$, $p_i$ divides $p_{\tilde l+1}\cdots p_l\frac{q+1}{d}$ but not $\ov{y}$. Therefore $(y,t)=1$, moreover $y\equiv\ov{y}\mod \frac{q+1}{d}$ and $yn\equiv d\mod{q+1}$.
	
	\begin{step}\label{u:st2}
		We construct matrices $A,X\in\gu{n}{q}$ \st $\det{A}=\omega^k$ and $[A,\phi^s X]\in Z$. 
	\end{step}
	Since $t\mid n$ and $t\neq1,n$, we have that both $t\geq2$ and $n-t\geq2$  and so we can define  
	\[
	A:=\begin{pmatrix}
	A_{t,y} & 0\\
	0 & A_{n-t,k-y}
	\end{pmatrix}\in\gu{n}{q}.
	\]
	
	First, notice that 
	\[
	\det{A}=\det{A_{t,y}}\det{A_{n-t,k-y}}=\omega^y\omega^{k-y}=\omega^k,
	\]
	therefore $\rho(AZ)=\delta^k$.
	
	Let $r:=(p^s-1)/t$ and define
	
	\[
	X:=\begin{pmatrix}
	X_{t,yr} & 0\\
	0 & X_{n-t,yr}
	\end{pmatrix}\in\gu{n}{q},
	\]
	
	We have that 
	
	\[
	y(p^s-1)=yrt \mod{q+1},
	\]
	
	so applying Lemma \ref{ublock} with $(w,l,c)=(t,y,yr)$ we get
	
	\[
	A_{t,y}^{\phi^s X_{t,yr}}=\omega^{yr}A_{t,y}.
	\]
	
	Moreover, recalling that $kt=d\equiv ny \mod{q+1}$, we have that 
	
	\[
	(k-y)(p^s-1)=(k-y)rt=krt-yrt=yr(n-t) \mod{q+1},
	\]
	
	so applying Lemma \ref{ublock} with $(w,l,c)=(n-t,k-y,yr)$ we get
	
	\[
	A_{n-t,k-y}^{\phi^s X_{n-t,yr}}=\omega^{yr}A_{n-t,k-y}.
	\]

	Therefore we have
	\[
	A^{\phi^s X}=\omega^{yr}A
	\]
	or, equivalently,
	\[
	[A,\phi^s X]\in Z.
	\]
	
	\begin{step}\label{u:st3}
		We find a matrix  $C\in\gu{n}{q}$ \st $[A,C]=1$ with $\det{C}=\omega$.
	\end{step}

	Recall that $(y,t)=1$, so there exist $a,b\in\mathbb{Z}$ \st $ay+bt=1$.
	Let 
	\[
	C_0:=\omega^b A_{t,y}^a \in\gu{t}{q}.
	\]
	We have that $[A_{t,y},C_0]=1$ and 
	\[
	\det{C_0}=\det{A_{t,y}}^a\omega^{bt}=\omega^{ay+bt}=\omega.
	\]
	Let 
	\[
	C:=\begin{pmatrix}
	C_0 & 0\\
	0 & 1
	\end{pmatrix}\in\gu{n}{q}.
	\]
	We have $[A,C]=1$ and $\det{C}=\det{C_0}=\omega$.
	\begin{step}
		We complete the proof by constructing a $T$-abelian supplement.
	\end{step}
	
	Let $u\in\mathbb{Z}$ \st $\rho(XZ)=\delta^u$. 
	Combining Steps \ref{u:st2} and \ref{u:st3}, we get
	\[
	[A,\phi^s XC^{j-u}]\in Z,
	\]
	with $\rho(AZ)=\delta^k$ and $\rho(XC^{j-u}Z)=\delta^u\delta^{j-u}=\delta^j $.
	Therefore
	\[
	\tilde T:=\gen{A,\phi^s  XC^{j-u}}Z/Z
	\]
	is a $T$-abelian supplement.
\end{proof}

By Theorem \ref{th_comp} and Table \ref{table_d}, to prove Theorem \ref{almost} we are left to deal with the following cases:
$$
B_n(q),C_n(q) , D_n(q), E_7(q),\ q=p^m, p\not=2
$$
$$
{}^2D_n(q),\ q=p^m, p\not=2, n\ \text{even}
$$
$$
E_6(q),\ q=p^m, q\equiv 1 \mod 3
$$
and
$$
{}^2\!E_6(q),\ q=p^m, q\equiv -1 \mod 3
$$
In the next sections we shall deal with these cases. We start by briefly recalling the notation we shall use.

\section{Notation for groups of Lie type}\label{sec_lie}

For the definitions and automorphisms of simple groups of Lie type we refer  to \cite{Carter} (see also \cite{Steinberg}). We briefly recall that
the Chevalley group (or untwisted group of Lie type) $L(q)$, viewed as a group of automorphisms of a Lie algebra $L_k$ over the field $k=\mathbb{F}_q$ of characteristic $p$, obtained from a complex finite dimensional simple Lie algebra $L$, is the group generated by certain automorphisms $x_r(t)$, where $t$ runs over $\mathbb{F}_q$ and $r$ runs over the root system $\Phi$ associated to $L$. For every $r\in \Phi$, $t\in k^\times$, one defines 
$n_r(t)=x_r(t)x_{-r}(-t^{-1})x_r(t)$, $n_r=n_r(1)$ and the subgroup $N=\<{n_r(t)\mid r\in \Phi, t\in k^\times}$ of $L(q)$.

Let $\Delta=\{\alpha_1,\ldots,\alpha_n\}$ be a system of simple roots of $\Phi$. We shall use the numbering and the description of the simple roots in terms of the canonical basis $(e_1,\ldots, e_k)$ of an appropriate $\mathbb{R}^k$ as in \cite{bourbaki}, Planches I-IX.
We denote by $Q$ the root lattice, by $P$  the weight lattice and by $W$ the Weyl
group; $s_i$ is the simple reflection associated to $\alpha_i$,
$\{\om_1,\ldots,\om_n\}$ are the fundamental weights, $w_0$ is the longest element of $W$, 
$A = (a_{ij})$ is the
Cartan matrix (hence $\alpha_i=\sum_j a_{ij}\om_j$). 

Let $\Hom(Q,\mathbb{F}_q^\times)$ be the group of $\mathbb{F}_q$-characters of $Q$ (i.e. group homomorphisms of $Q$ into $\mathbb{F}_q^\times$). For any $\chi\in \Hom(Q,\mathbb{F}_q^\times)$, one defines the automorphism $h(\chi)$ of $L_k$. Let $\hat H=\{h(\chi)\mid \chi\in \Hom(Q,\mathbb{F}_q^\times)\}$. The map $\chi\mapsto h(\chi)$ is an isomorphism of $\Hom(Q,\mathbb{F}_q^\times)$ onto $\hat H$. We have $\hat H\leq N_{\Aut L_k}(L(q))$. Let $H=\hat H\cap L(q)$. Then $h(\chi)$ lies in $H$ if and only if $\chi$ can be extended to an $\mathbb{F}_q$-character of $P$. The number $d$ in Table \ref{table_d} relative to the untwisted case is the order of $\hat H/H$. We have $H\triangleleft N$ and $N/H\cong W$. For $w\in W$, we denote by $\dot w$ a representative of $w$ in $N$; for each $i=1,\ldots,n$, $n_{\alpha_i}$ is a representative of $s_i$ in $N$. For short we denote
$n_{\alpha_i}$ by $n_i$. Note that $n_i$ lies in $L(p)$, so that it is fixed by field automorphisms of $L(q)$.

We give a short description of the twisted groups. Assume that the Dynkin diagram of $L$ has a non-trivial symmetry $\tau$ of order $s$. We stick to the cases $D_n$, $E_6$ and $s=2$ since we do not need to deal with ${}^3\!D_4(q)$, ${}^2\!B_2(2^r)$, ${}^2\!F_4(2^r)$ and ${}^2G_2(3^r)$. One defines the twisted group ${}^2\!L(q)$ as a certain subgroup of the Chevalley group $L(q^2)$. 
Let $E$ be the real vector space spanned by the roots (or the weights). Then $\tau$ induces an automorphism (in fact an isometry) $\sigma$ of $E$ fixing both $Q$ and $P$. Let $\chi$ be an $\mathbb{F}_{q^2}$-character of $Q$ (or $P$). We say that $\chi$ is {\it self-conjugate} if $\chi(\tau(x))=\chi(x)^q$ for every $x$ in $Q$ (or $P$). Let
$\hat H^1=\{h(\chi)\mid  \chi~:~Q~\to~\mathbb{F}_{q^2}^\times  \ \text{is a self-conjugate character of} \ Q\}$. We have $\hat H^1\leq N_{\Aut L_{\mathbb{F}_{q^2}}}({}^2\!L(q))$. Let $H^1=\hat H^1\cap {}^2\!L(q)$. Then $h(\chi)$ lies in $H^1$ if and only if $\chi$ can be extended to a self-conjugate $\mathbb{F}_{q^2}$-character of $P$. The number $d$ in Table \ref{table_d} relative to the twisted case is the order of $\hat H^1/H^1$.

In general we have $(1-z)P\leq Q$ for every $z\in W$. For Coxeter elements equality holds:
\begin{lemma}\label{Coxeter}
 Let $\alpha_1,\ldots\alpha_n$ be the simple roots (in any fixed order), $\om_1\ldots,\om_n$ the corresponding fundamental weights. Then 
$$
(1-s_1\cdots s_n)\om_i=\alpha_i+z_1\alpha_1+\cdots +z_{i-1}\alpha_{i-1}
$$
with $z_1,\ldots,z_{i-1}\in \mathbb{Z}$. In particular $(1-s_1\cdots s_n)P=Q$.
\end{lemma}
\begin{proof} We have $s_i(\om_j)=\om_j-\delta_{ij}\alpha_i$ for every $i$, $j$.
For $i=1$ we have $s_1\cdots s_n\om_1=s_1\om_1=\om_1-\alpha_1$, hence $(1-s_1\cdots s_n)\om_1=\alpha_1$. Let $1<i\leq n$. Then $s_1\cdots s_{i-1}(\alpha_i)=\alpha_i+z_1\alpha_1+\cdots+z_{i-1}\alpha_{i-1}$, with $z_k\in \mathbb{Z}$ for $k=1,\ldots, i-1$. Then
$$
(1-s_1\cdots s_n)\om_i=
\om_i-s_1\cdots s_i\om_i=
\om_i-s_1\cdots s_{i-1}(\om_i-\alpha_i)=
$$
$$
=\om_i-(\om_i-s_1\cdots s_{i-1}\alpha_i)=s_1\cdots s_{i-1}\alpha_i=\alpha_i+z_1\alpha_1+\cdots+z_{i-1}\alpha_{i-1}.\qedhere
$$
\end{proof}

Let $\chi$ be a character of $Q$, $w$ in $W$.
We define the character $w\chi$ in the following way. Let $x\in Q$: we put $w\chi(x):=\chi(w^{-1}x)$, i.e. $w\chi=\chi\circ w^{-1}$.  We also define
$\tau\chi$, where $\tau$ is a graph-automorphism, by $(\tau\chi)x:=\chi(\tau^{-1}x)$ for $x\in Q$ (hence $\tau\chi=\chi\circ \tau^{-1}$). Note that for $w\in W$, we have (\cite[Theorem 7.2.2]{Carter})
$$
\dot w h(\chi)\dot w^{-1}=h(w\chi).
$$
Since we are assuming $\Phi$ of type $D_n$ or $E_6$, there is a Coxeter element $w$ in $W$ fixed by the graph-automorphisms. We may choose a representative $\dot w$ of $w$ in $N$ over the prime field and fixed by the graph-automorphisms. 
Let $F$ be a field or a graph-field automorphism of $L(q)$. Then $F$ fixes $\dot w$ and acts on $\hat H$, hence it induces an automorphism $g$ of $\Hom(Q,k^\times)$, $F(h(\chi))=h(g(\chi))$. Let $\chi:Q\to k^\times$ be a fixed character, $x=\dot w h(\chi)$. We shall look for an element $y=h(\chi')\in \hat H$ such that $[x,Fy]=1$, i.e.
$$
xF y=F yx\iff
y^{-1}F^{-1}xFy=x\iff
y^{-1}F(x)y=x
$$
so that
$$
h(\chi')^{-1}\dot w h(g(\chi))h(\chi')=\dot w h(\chi).
$$
We have
$h(\chi')^{-1}\dot w =\dot w\dot w^{-1}h(\chi')^{-1}\dot w=\dot wh(w^{-1}\chi')^{-1}=\dot wh(-w^{-1}\chi')$, hence
$$
\dot wh(-w^{-1}\chi') h(g(\chi))h(\chi')=\dot w h(\chi)
$$
$$
h(-w^{-1}\chi') h(g(\chi))h(\chi')=h(\chi)
$$
and finally
$$
h((1-w^{-1})\chi')=h((1-g)\chi)\quad,\quad  (1-w^{-1})\chi'=(1-g)\chi
$$
$$
\chi'\circ (1-w)=(1-g)\chi.
$$
We shall be interested in the following cases:

$F:s\mapsto s^{[p^i]}$, then 
$$
\chi'\circ (1-w)=(1-p^i)\chi,
$$

$F:s\mapsto s^{[p]\tau}$ then 
$$
\chi'\circ (1-w)=(1-p\tau)\chi.
$$
\vskip20pt

By Lemma \ref{Coxeter} we have $(1-w)^{-1}Q=P$.
Let $\Delta=|\ \! P/Q\!\ |=\det A$. Then $\Delta P\leq Q$. Note that if $\Phi=D_n$ with $n$ even, then $2 P\leq Q$ since $P/Q\cong C_2\times C_2$ (the inverses of the Cartan matrices my be explicitly found in \cite{WZ}).
We put $\Delta_1=|\ \! P/Q\!\ |$ unless $\Phi=D_n$, $n$ even, in which case we put $\Delta_1=2$. Then
$$
\Delta_1(1-w)^{-1}Q\leq Q
$$
and we may define the character 
$$
\zeta_\chi=\chi\circ \Delta_1(1-w)^{-1}:Q\to k^\times
$$
and $h(\zeta_\chi)\in \hat H$. 

We start with the cases $B_n(q),C_n(q), E_7(q)$.

\section{$C_n(q)$, $B_n(q)$, $n\geq 2$, $E_7(q)$}\label{sec_BCE7}

Here $L$ is of type $C_n$, $B_n$ or $E_7$, $G_0=L(q)$, $q=p^m$,  $d=(q-1,2)$ and we assume that $\aut(G_0)$ does not split over $G_0$, $(\frac{q-1}d,d,m)\not=1$. Therefore $d=2$ and $p$ is odd:
$$
\out(G_0)=\gen{\delta}\times\gen{\phi},
$$
$|\ \! \delta\!\ |=2$, $|\ \! \phi\!\ |=m$. We fix an $\mathbb{F}_q$-character $\chi$ of $Q$ which can not be extended to a character of $P$, so that $h(\chi)$ induces $\delta$ in $\out(G_0)$. We look for an $\mathbb{F}_q$-character $\chi'$ so that
$[\dot w h(\chi),\phi h(\chi')]=1$, i.e.
$$
\chi'\circ (1-w)=(1-p)\chi.
$$

We have $\Delta_1=2$, so  $\zeta_\chi=\chi\circ 2(1-w)^{-1}$. We take
$$
\chi'=\frac{1-p}2\ \zeta_\chi
$$
so $h(\chi')=h(\zeta_\chi)^{\frac{1-p}2}$.
Therefore
$$
\tilde T=\gen{\dot w h(\chi),\phi h(\chi')}
$$
is an $\out(G_0)$-abelian supplement (arguing as in the $\psl{2}{q}$ case).

We have proved
\begin{thm}\label{CnBnE7}
Let $G$ a finite almost simple group with socle $G_0=C_n(q), B_n(q)$ or $E_7(q)$. Then $G$ contains an abelian subgroup $A$ such that $G=AG_0$.
\end{thm}

\section{$E_6(q)$}\label{sec_E6}

Here $L$ is of type $E_6$, $G_0=L(q)$, $q=p^m$,  $d=(q-1,3)$ and we assume that $\aut(G_0)$ does not split over $G_0$, $(\frac{q-1}d,d,m)\not=1$. Therefore $d=3$ and $p\not=3$:
$$
\out(G_0)=\gen{\delta}\sd\gen{\phi,\tau}
$$
$|\ \! \delta\!\ |=3$, $|\ \! \phi\!\ |=m$, $\delta^\phi=\delta^p$, $\delta^\tau=\delta^{-1}$, $[\phi,\tau]=1$.
We fix an $\mathbb{F}_q$-character $\chi$ of $Q$ which can not be extended to a character of $P$.

Let $\pi\colon\out(G_0)\to\out(G_0)/\gen{\delta}=\gen{\phi,\tau}$. Let $T$ be a non-cyclic abelian subgroup of $\out(G_0)$.
If $\pi(T)$ is not cyclic, then $\pi(T)=\<{\phi^s,\tau}$. Therefore $T= \<{\phi^s\delta^i,\tau\delta^k}$. But $\tau\delta^k$ is conjugate to $\tau$ under $\<\delta$, hence we may assume $T= \<{\phi^s\delta^i,\tau}$, so $T= \<{\phi^s,\tau}\leq \<{\phi,\tau}$ and so $T$ is itself a $T$-abelian supplement.

We are left with $\pi(T)$ is cyclic, $
\pi(T)=\<{\phi^s\tau^\epsilon}
$. Then $T=\<{\delta,\phi^s\tau^\varepsilon}$ and
we distinguish the two cases:

$p\equiv 1 \mod 3$.
Since $[\delta,\phi]=1$, we get 
$\varepsilon=0$, $T\leq\<{\delta,\phi}$, so it is enough to consider
$$
p\equiv 1 \mod 3\ ,\ T= \<{\delta,\phi}\quad\quad\text{case 1}
$$

$p\equiv -1 \mod 3$.
Let $\varepsilon=1$, $T=\<{\delta,\varphi^s\tau}$.  Since $[\delta,\varphi\tau]=1$, $s$ must be odd. Therefore
$T\leq \<{\delta,\varphi\tau,\varphi^2}=\<{\delta,\varphi\tau}$. It is enough to consider
$$
p\equiv -1 \mod 3\ ,\ T=\<{\delta,\varphi\tau}\quad\quad\text{case 2}
$$
If $\varepsilon=0$, $T=\<{\delta,\varphi^s}$, so $s$ is even and again $T\leq \<{\delta,\varphi^2} < \<{\delta,\varphi\tau}$.

\noindent
Summarising, we only have to deal with cases 1, 2.

We consider the Coxeter element $w=s_1s_4s_6s_3s_2s_5$, fixed by the graph-automorphism $\tau$. We choose a representative $\dot w$ of $w$ in $N$ over the prime field and fixed by $\tau$, $\dot w=n_1n_4n_6n_3n_2n_5$ for instance. Hence $\tau \dot w=\dot w\tau$, $\dot w\phi=\phi\dot w$. Here $\phi$ is the field automorphism of $G_0$ sending $x$ to $x^{[p]}$. We use the notation $\phi^{-1}x\phi=x^{[p]}$. We have $\Delta_1=3$, so  $\zeta_\chi=\chi\circ 3(1-w)^{-1}$. 
\subsection{$p\equiv 1\mod 3$, $T=\<{\delta,\varphi}$}
\vskip10pt
We take
$$
\chi'=\frac{1-p}3\ \zeta_\chi
$$
so $h(\chi')=h(\zeta_\chi)^{\frac{1-p}3}$.
Therefore
$$
\tilde T=\<{\dot w h(\chi),\phi h(\chi')}
$$
is a $T$-abelian supplement.

\subsection{$p\equiv -1\mod 3$, $T=\<{\delta, \varphi\tau}$}
\vskip10pt
Since $\tau w_0=-1$, we have
$$
(1+\tau)P=(1+\tau)w_0P=(w_0+\tau w_0)P=(w_0-1)P=(1-w_0)P\leq Q
$$
hence, by Lemma \ref{Coxeter}
$$
(1+\tau)(1-w)^{-1}Q=(1+\tau)P\leq Q
$$
so $\chi\circ(1+\tau)(1-w)^{-1}$ is an $\mathbb{F}_q$-character of $Q$.
We look for an $\mathbb{F}_q$-character $\chi'$ so that
$[\dot w h(\chi),\phi\tau h(\chi')]=1$, i.e.
$$
\chi'\circ (1-w)=(1-p\tau)\chi
$$
We have $1-p\tau=1+p-p-p\tau=1+p-p(1+\tau)$, and we may define
$$
\chi'=\frac{1+p}3\ \zeta_\chi-p\,\chi\circ (1+\tau)(1-w)^{-1}
$$
obtaining a character which isatisfies $\chi'\circ (1-w)=(1-p\tau)\chi$. Therefore
$$
\tilde T=\<{\dot w h(\chi),\phi\tau h(\chi')}
$$
is a $T$-abelian supplement.	

We have proved
\begin{thm}\label{E6}
Let $G$ a finite almost simple group with socle $G_0=E_6(q)$. If $G/G_0$ is abelian, then $G$ contains an abelian subgroup $A$ such that $G=AG_0.$
\end{thm}

\section{${}^2\!E_6(q)$}\label{sec_2E6}	

Here $L$ is of type $E_6$, $G_0={}^2\!E_6(q)\leq E_6(q^2)$, $q=p^m$,  $d=(q+1,3)$ and we assume that $\aut(G_0)$ does not split over $G_0$, $(\frac{q+1}d,d,m)\not=1$. Therefore $d=3$ and $q\equiv -1 \mod 3$, so $p\equiv -1 \mod 3$ and $m$ is odd:
$$
\out(G_0)=\gen{\delta}\sd\gen{\phi}
$$
$|\ \! \delta\!\ |=3$, $|\ \! \phi\!\ |=2m$, $\delta^\phi=\delta^{-1}$.

It is enough to consider the case $T=\< {\delta, \phi^{2}}$.
We fix a self-conjugate $\mathbb{F}_{q^2}$-character $\chi$ of $Q$ which can not be extended to a self-conjugate $\mathbb{F}_{q^2}$-character of $P$ (so that $h(\chi)\in \hat H^1\setminus H^1$).

We consider the same Coxeter element $w=s_1s_4s_6s_3s_2s_5$ as in the previuos section, and the same representative $\dot w=n_1n_4n_6n_3n_2n_5$, which lies in $G_0$.

We look for an element $h(\chi')\in \hat H^1$ so that $[\dot w h(\chi),\phi^2h(\chi')]=1$, i.e.
$$
\chi'\circ (1-w)=(1-p^2)\chi.
$$

We have $\Delta_1=3$, so  $\zeta_\chi=\chi\circ 3(1-w)^{-1}$. We take
$$
\chi'=\frac{1-p^2}3\ \zeta_\chi
$$
so $h(\chi')=h(\zeta_\chi)^{\frac{1-p^2}3}$.

Note that since $\chi$ is self-conjugate and $\tau w=w\tau$, $\zeta_\chi$ and $\chi'$ are self-conjugate, so $h(\chi')$ lies in $\hat H^1$.
Therefore
$$
\tilde T=\gen{\dot w h(\chi),\phi^2 h(\chi')}
$$
is a $T$-abelian supplement.

We have proved
\begin{thm}\label{E6tw}
Let $G$ a finite almost simple group with socle $G_0={}^2\!E_6(q)$. If $G/G_0$ is abelian, then $G$ contains an abelian subgroup $A$ such that $G=AG_0.$
\end{thm}

\section{${}^2\!D_n(q)$, $n$ even}	\label{sec_2D}

Here $L$ is of type $D_n$, $n$ even, $G_0={}^2\!D_n(q)\leq D_n(q^2)$, $q=p^m$,  $d=(q+1,2)$ and we assume that $\aut(G_0)$ does not split over $G_0$, $d\not=1$. Therefore $d=2$ and $p\not=2$,
$$
\out(G_0)=\gen{\delta}\times\gen{\phi}
$$
$|\ \! \delta\!\ |=2$, $|\ \! \phi\!\ |=2m$.

It is enough to consider the case $T=\out(G_0)$.
We fix a self-conjugate $\mathbb{F}_{q^2}$-character $\chi$ of $Q$ which can not be extended to a self-conjugate $\mathbb{F}_{q^2}$-character of $P$ (so that $h(\chi)\in \hat H^1\setminus H^1$).

We consider the Coxeter element $w=s_1s_2\cdots s_{n-1}s_n$, fixed by $\tau$ (which exchanges $\a_{n-1}$ and $\a_n$), and the representative $\dot w=n_1n_2\cdots n_{n-1}n_n$, which lies in $G_0$.
We look for an element $h(\chi')\in \hat H^1$ so that $[\dot w h(\chi),\phi h(\chi')]=1$, i.e.
$$
\chi'\circ (1-w)=(1-p)\chi
$$
We have $\Delta_1=2$ (since $n$ is even), so  $\zeta_\chi=\chi\circ 2(1-w)^{-1}$. We take
$$
\chi'=\frac{1-p}2\ \zeta_\chi
$$
so $h(\chi')=h(\zeta_\chi)^{\frac{1-p}2}$.

Since $\chi$ is self-conjugate and $\tau w=w\tau$, $\zeta_\chi$ and $\chi'$ are self-conjugate, so $h(\chi')$ lies in $\hat H^1$.
Therefore
$$
\tilde T=\gen{\dot w h(\chi),\phi h(\chi')}
$$
is an $\out(G_0)$-abelian supplement.

We have proved
\begin{thm}\label{Dntw}
Let $G$ a finite almost simple group with socle $G_0={}^2\!D_n(q)$. Then $G$ contains an abelian subgroup $A$ such that $G=AG_0$.
\end{thm}

In the next sections we shall deal with the remaining cases: $D_n(q)$, $q=p^m$. We shall use the identifications with classical groups as in \cite[Theorem 11.3.2]{Carter} and \cite[1.11, 1.19]{Carter2}. Here $\lambda$ is a generator of $\F_q^\times$.

We have $G_0=P\Omega_{2n}^+(q)$, ${\rm Inndiag}(G_0)=P(CO_{2n}(k)^\circ)$, where $CO_{2n}(k)$ if the conformal orthogonal group, i.e. the group of orthogonal similitudes of $k^{2n}$, $k=\F_q$;  $CO_{2n}(k)^\circ$ is the subgroup of index 2 of $CO_{2n}(k)$ of elements which do not interchange the two  families of maximal isotropic subspaces of $k^{2n}$. If
$(e_1,\ldots,e_n,f_1,\ldots,f_n)$ is the canonical basis of $k^{2n}$, the bilinear form on $k^{2n}$ corresponds to the matrix 
$$
K_n=
\left(
\begin{array}{cccccc}
0_{n}&I_{n}\\
I_{n}&0_{n}\\
\end{array}
\right)
$$
We define the homomorphism $\eta:CO_{2n}(k)^\circ\to k^\times$ by
$$
\eta(X)=\mu \quad \text{if}\quad  {}^tXK_nX=\mu K_n
$$
For $\mu\in k^\times$, let $o_\mu=\left(
\begin{smallmatrix}
I_n&0\\
0&\mu I_n
\end{smallmatrix}
\right)$, so that $\eta(o_\mu)=\mu$.

The graph automorphism $\tau$ of $D_n$ exchanging $\a_{n-1}$ and $\a_n$ is induced by conjugation with
$$
\tau_n=
\left(
\begin{array}{cccccc}
I_{n-1}&0&0_{n-1}&0\\
0&0&0&1\\
0_{n-1}&0&I_{n-1}&0\\
0&1&0&0\\
\end{array}
\right)\in O_{2n}(q)
$$
$\tau_n^2=1$,
$x^\tau=\tau_n x\tau_n$.
$$
x_{\a_i}(z)^\t=x_{\a_i}(z), i=1,\ldots, n-2\quad ,\quad
x_{\a_{n-1}}(z)^\t=x_{\a_n}(z),
x_{\a_n}(z)^\t=x_{\a_{n-1}}(z)
$$
We shall deal with tha cases $n$ odd and even separately.

\section{$D_n(q)$, $n\geq 3$, $n$ odd}\label{sec_Dodd}

Here $L$ is of type $D_n$, $n$ odd, $G_0=D_n(q)$, $q=p^m$,  $d=(4 ,q-1)$ and we assume that $\aut(G_0)$ does not split over $G_0$, hence $( \frac{q^n-1}d,d,m)\not=1$. In particular, $d\not=1$, hence $p$ is odd and $d=2$ or $4$. Moreover $m$ is even, hence $4$ divides $q-1$. Therefore  $d=4$.
$$
\Out(G_0)=\<{\d, \tau, \varphi\mid \d^4=\tau^2=1, \d^\t=\d^{-1}, \varphi^m=[\t,\varphi]=1, \d^\varphi=\d^p}
$$
In $\Omega^+_{2n}(k)$ we choose
$$
\dot w_0=
\left(
\begin{array}{cccccc}
0_{n-1}&0&I_{n-1}&0\\
0&1&0&0\\
I_{n-1}&0&0_{n-1}&0\\
0&0&0&1\\
\end{array}
\right)
$$
a representative of the longest element $w_0$ of the Weyl group.
We have $\dot w_0^2=1$, $\dot w_0\tau_n=\tau_n\dot w_0=K_n$.
Let $X\in CO_{2n}(k)^\circ$, $\eta(X)=\mu$, i.e. ${}^tXK_nX=\mu K_n$. Then
${}^tX=\mu K_nX^{-1}K_n$, so that
\begin{equation}\label{generico}
{}^tX^{-1}=\eta(X)^{-1}\dot w_0 \tau_n X\tau_n \dot w_0=\eta(X)^{-1}\dot w_0  X^\tau \dot w_0
\end{equation}
We start with $D_3$, exploiting the fact that $D_3=A_3$.
Let $k=\F_q$, $K=\overline k$, $V=k^4$ with canonical basis ${\ca B}=(v_1,\ldots,v_4)$, $\overline V=K^4$ with the same basis. Let
$$
\sigma:GL(\overline V)\to GL(\wedge^2 \overline V)\quad,\quad f\mapsto \wedge^2 f 
$$
We choose the basis $\ca B$ for $\overline V$, and the basis ${\ca C}=(v_{12}, v_{13}, v_{23}, v_{34}, v_{42}, v_{14})$, where $v_{ij}=v_i\wedge v_j$, for $\wedge^2 \overline V$. We endow $\wedge^2\overline V$ with the symmetric bilinear form with matrix $K_3$
with respect to ${\ca C}$. Then $\sigma(GL(\overline V))\leq CO(\wedge^2\overline V)^\circ$, $\sigma(GL(V))\leq CO(\wedge^2V)^\circ$ and, by considering bases, we obtain the homomorphism $\sigma:GL_4(k) \to CO_6(k)^\circ$.
 We have
  $$
 \sigma:
\left(
\begin{array}{cc}
I_3&0\\
0&\mu
\end{array}
\right)\mapsto
\left(
\begin{array}{cc}
I_3&0_3\\
0_3&\mu I_3
\end{array}
\right)=o_\mu
$$
in particular
$$
\det \left(
\begin{array}{cc}
I_3&0\\
0&\mu
\end{array}
\right)=\mu=\eta(o_\mu)
$$
Moreover
$
 \sigma:
\mu I_4
\mapsto
\mu^2 I_6
$.
If $X\in GL_4(k)$, $\det X=\mu$, then 
$
X=Y\left(
\begin{smallmatrix}
I_3&0\\
0&\mu
\end{smallmatrix}
\right)
$
with $Y\in SL_4(k)$,
$
\sigma(X)=\sigma(Y)o_\mu$
with $\sigma(Y)\in \Omega_6^+(k)$ (\cite[Theorem 12.20]{Taylor}), hence
\begin{equation}\label{det}
\eta(\sigma(X))=\mu=\det X
\end{equation}
From (\ref {generico}) and (\ref {det}) we get
\begin{equation}\label{comparison}
\sigma({}^tX^{-1})={}^t(\sigma(X))^{-1}=(\det X)^{-1}\dot w_0\sigma(X)^\tau\dot w_0
\end{equation}
For $X$, $Y\in GL_4(k)$, $z\in k^\times$
we get
\begin{equation}\label{equival}
\begin{split}
Y^{-1}X^{[p]} Y = z  X
\quad &\Rightarrow\quad
\sigma(Y)^{-1}\sigma(X)^{[p]} \sigma(Y)= z^2 \sigma(X)\\
Y^{-1}({}^tX^{-1} )Y = z  X
\quad &\Rightarrow\quad
Z^{-1}  \sigma(X)^\tau Z= z^2\det (X) \sigma(X)\ , \ Z=\dot w_0\sigma(Y )
\\
{}^tX^{-1} Y = z Y^{[p]} X
\quad &\Rightarrow\quad
\sigma(X)^\tau Z= z^2 \det(X) Z^{[p]}\sigma(X)\ ,\  Z=\dot w_0\sigma(Y )
\end{split}
\end{equation}
since $\dot w_0^{[p]}=\dot w_0$.

In section \ref{sec_lin}, for a given abelian subgroup $T$ of $\out(\psl{4}{q})$ we have exhibited a $T$-abelian supplement $\tilde T$ by giving matrices in $\gl{4}{q}$: 
the map $\sigma$ allows to solve the problem for $G_0=P\Omega^+_6(q)$, by giving matrices in $CO_{6}(k)^\circ$. Now we consider $D_n$, $n$ odd, $n=1+2m$, $n>3$. The space $k^{2n}$ is the orthogonal direct sum
$k^{2n}=U\oplus U^\perp$, where $U=\<{e_1,\ldots,e_{n-3},f_1,\ldots,f_{n-3}}$, $U^\perp=\<{e_{n-2},e_{n-1},e_n,f_{n-2},f_{n-1},f_n}$, with $\dim U=2n-6=4(m-1)$.
Moreover $U$ is the direct orthogonal sum of subspaces of dimension 4:
$$
U_1=\<{e_1,e_2,f_1,f_2},\ldots, U_{m-1}=\<{e_{n-4}, e_{n-3},f_{n-4},f_{n-3}}
$$
To define an isometry or more generally an orthogonal similitude of $k^{2n}$ we may give matrices 
$X_i\in CO_4(q)^\circ$, $\eta(X_i)=\mu$, $i=1,\ldots,m-1$, $X\in CO_6(q)^\circ$, $\eta(X)=\mu$ and define $Y$ in ${\rm GL}_{2n}(q)$ by
$$
Y=X_1\oplus \cdots \oplus X_{m-1}\oplus X
$$
Then $Y\in CO_{2n}(q)^\circ$, with $\eta(Y)=\mu$. If $Y\in CO_{2n}(q)^\circ$ fixes $U^\perp$, then it fixes $U$ and if we write $Y=X\oplus Z$, with $X\in CO_6(q)^\circ$, $Z\in CO_{2n-6}(q)^\circ$, and consider the action of $\varphi$ and $\tau$, we get
$$
Y^{[p]}=X^{[p]}\oplus Z^{[p]}\quad,\quad
Y^\tau=X^\tau \oplus Z
$$
since $\tau_n$ acts on the basis $(e_1,\ldots,e_n,f_1,\ldots,f_n)$ just switching $e_n$ and $f_n$ (here $Y^\tau =\tau_n Y\tau_n$,
$X^\tau = \tau_3 X\tau_3$).

We shall proceed as follows. Assume $T$ is an abelian subgroup of $\Out (G_0)$. We consider the analogous subgroup $T$ of $\Out (\psl{4}{q})$. From the $\psl{4}{q}$ case, we have an abelian subgroup of $\aut (\psl{4}{q})$ given by explicit matrices in ${\rm GL}_4(q)$. By using $\sigma$ we obtain corresponding matrices in $CO_6(q)^\circ$ satisfying certain relations. For each such matrix $X$ we define a matrix $X_1\in CO_4(q)^\circ$ and finally define the matrix $Y=X_1\oplus\cdots\oplus X_1\oplus X$ in $CO_{2n}(q)^\circ$ ($m-1$ copies of $X_1$). We shall then obtain a $T$-abelian supplement $\tilde T$ in $\Aut (G_0)$.

Let $A$, $B\in GL_2(k)$ with 
$$
B^{-1}A^{[p]}B=z A\quad , \quad \det A=\mu, \ z=\mu^{\frac12(p-1)}
$$
and let $\nu\in k^\times$.
Our aim is to define orthogonal similitudes of $k^4$ (with respect to the form given by $K_2$). We put
$$
a=a(A)=
\left(
\begin{array}{cc}
A & 0_2 \\
 0_2 & {}^tA^{-1} \\
\end{array}
\right)\left(
\begin{array}{cc}
I_2& 0_2 \\
 0_2 &(\det A) I_2 \\
\end{array}
\right)\in CO_4(q)^\circ \quad,\quad
\eta(a)=\det A
$$
$$
b=b(B,\nu)=
\left(
\begin{array}{cc}
B & 0_2 \\
 0_2 & {}^tB^{-1} \\
\end{array}
\right)\left(
\begin{array}{cc}
I_2& 0_2 \\
 0_2 & \nu I_2 \\
\end{array}
\right)\in CO_4(q)^\circ \quad,\quad
\eta(b)=\nu
$$
From $
B^{-1}A^{[p]}B=\mu^{\frac12(p-1) }A
$
we get
$$
b^{-1}a^{[p]}b=\mu^{\frac12(p-1)} a
\quad,\quad \eta(a)=\det A=\mu, \eta(b)=\nu
$$
We shall take
$$
A=\left(
\begin{array}{cc}
0 & -\mu \\
 1 & 0 \\
\end{array}
\right)\quad,\quad
B=\left(
\begin{array}{cc}
\mu^{\frac12(p-1)}& 0 \\
 0 & 1 \\
\end{array}
\right)\quad,\quad
B^{-1}A^{[p]}B=\mu^{\frac12(p-1)} A
$$
\begin{equation}\label{A}
a=a(A)=a(\mu)=\left(
\begin{array}{cccc}
 0 & -\mu& 0 & 0 \\
 1 & 0 & 0 & 0 \\
 0 & 0 & 0 & -1 \\
 0 & 0 & \mu & 0 \\
\end{array}
\right)\
, \ \eta(a)=\mu
\end{equation}
\begin{equation}\label{B}
b=b(B,\nu)=b(\mu,\nu)=
\left(
\begin{array}{cccc}
 \mu^{\frac{1}{2}(p-1)} & 0 & 0 & 0 \\
 0 & 1 & 0 & 0 \\
 0 & 0 & \mu^{-\frac{1}{2}(p-1)} & 0 \\
 0 & 0 & 0 & 1 \\
\end{array}
\right)
\left(
\begin{array}{cccc}
 1 & 0 & 0 & 0 \\
 0 & 1 & 0 & 0 \\
 0 & 0 & \nu & 0 \\
 0 & 0 & 0 & \nu \\
\end{array}
\right)
\ ,\
\eta(b)=\nu
\end{equation}
Then
$$
b^{-1}a^{[p]}b=\mu^{\frac12(p-1)} a\quad,\quad \eta(a(\mu))=\mu,\ \eta(b(\mu,\nu))=\nu
$$
Note that for any $i\in {\mathbb Z}$ we have
$$
B^{-1}(A^i)^{[p]}B=\mu^{\frac12i(p-1)} A^i\quad,\quad \det A^i=\mu^i
$$
$$
b(\det A,\nu)^{-1}a(A^i)^{[p]}\,b(\det A,\nu)=\mu^{\frac12i(p-1)} a(A^i)
\quad,\quad \eta(a(A^i))=\det A^i=\mu^i, \ \eta(b)=\nu
$$
We shall make use of the explicit matrices in $\gl{4}{q}$ from section \ref{sec_lin}.

\subsection{$p\equiv1 \mod{4}$}
{ {\vskip  0mm plus .5mm } \noindent }

$T=\gen{\delta,\phi}$. In the $\psl{4}{q}$ case
 we took
$$
L=\begin{pmatrix}
0 & 0 & 0 & -\lambda\\
1 & 0 & 0 & 0\\
0 & 1 & 0 & 0\\
0 & 0 & 1 & 0
\end{pmatrix}, 
M=\begin{pmatrix}
\lambda^{\frac{3(p-1)}{4}} & 0 & 0 & 0\\
0 & \lambda^{\frac{2(p-1)}{4}} & 0 & 0\\
0 & 0 & \lambda^{\frac{p-1}{4}} & 0\\
0 & 0 & 0 & 1
\end{pmatrix}
$$
$$
\tilde T=\<{L,\varphi M}Z(\gl{4}{q})/Z(\gl{4}{q})
$$
We have
$
M^{-1}L^{[p]} M=\lambda^{\frac{p-1}{4}}L
$
hence in $CO_6(k)^\circ$, with $\ell=\sigma(L)$, $m=\sigma(M)$,
by (\ref{det}), (\ref{equival}):
$$
m^{-1}\ell^{[p]} m= \lambda^{\frac{p-1}{2}}\ell
$$
$$
\eta(\ell)=\det L=\lambda\quad,\quad
\eta(m)=\det M=\lambda^{\frac{3 (p-1)}{2}}
$$
We look for $a$, $b\in CO_4(k)^\circ$ satisfying the same relations using the above procedure. We take $\mu=\lambda$,  $\nu=\lambda^{\frac{3 (p-1)}{2}}$, $a=a(\lambda)$, $b=b(\lambda,\lambda^{\frac{3 (p-1)}{2}})$: if we put
$A_1=a\oplus\cdots \oplus a \oplus \ell$,
$B_1 =b\oplus\cdots\oplus b\oplus m$
then
$$
A_1, B_1\in CO_{2n}(q)^\circ, 
B_1^{-1}A_1^{[p]}B_1= \lambda^{\frac{p-1}{2}} A_1
$$
and
$$
\tilde T=\<{A_1,\varphi B_1}Z(CO_{2n}(k)^\circ)/Z(CO_{2n}(k)^\circ)
$$
is a $T$-abelian supplement.

 $T=\gen{\delta^2,\phi,\tau}$.
In the $\psl{4}{q}$ case for
$\gen{\delta^2,\phi,\gamma}$ we took
\[
L=\begin{pmatrix}
	0 & -\lambda & 0 & 0\\
	1 & 0 & 0 & 0\\
	0 & 0 & 0 & -\lambda\\
	0 & 0 & 1 & 0
	\end{pmatrix},
	M=\begin{pmatrix}
	\lambda^{\frac{p-1}{2}} & 0 & 0 & 0\\
	0 & 1 & 0 & 0\\
	0 & 0 & \lambda^{\frac{p-1}{2}} & 0\\
	0 & 0 & 0 & 1
	\end{pmatrix},
	N=\begin{pmatrix}
	\lambda^{-1} & 0 & 0 & 0\\
	0 & 1 & 0 & 0\\
	0 & 0 & \lambda^{-1} & 0\\
	0 & 0 & 0 & 1
	\end{pmatrix}
\]
$$
\tilde T=\gen{L,\phi
	M,\gamma
	N}Z/Z
$$
with
$$
M^{-1}L^{[p]}M = z_1 L\ ,\ 
N^{-1}({}^tL^{-1})N = z_2 L\ , \
({}^tM^{-1})N=z_3 N^{[p]} M
$$
$$
z_1=\lambda^{\frac12(p-1)}, z_2= \lambda^{-1}, z_3=1, \det L=\lambda^2, \det M=\lambda^{p-1}, \det N=\lambda^{-2}
$$
Hence in $CO_6(k)^\circ$, with $\ell=\sigma(L)$, $m=\sigma(M)$, $n=\dot w_0\sigma(N)$:
$$
m^{-1}\ell^{[p]}m = \lambda^{p-1} \ell\ , \ 
n^{-1}\ell^\tau n =   \ell\ , \ 
m^\tau n= \lambda^{p-1} n^{[p]} m
$$
$$
\eta(\ell)=\det L=\lambda^2,
\eta(m)=\det M=\lambda^{p-1},
\eta(n)=\eta(\dot w_0)\eta(\sigma(N))=\det N=\lambda^{-2}
$$
Recall that $\tau_n$ acts trivially on $U$, hence we have to define matrices $a$, $b$, $c\in CO_4(q)^\circ$ such that
$$
b^{-1}a^{[p]}b = \lambda^{p-1}a\ , \ 
c^{-1}a c =  a\ , \ 
b c= \lambda^{p-1} c^{[p]} b\quad,\ \text {i.e.} \quad b^{-1}c^{[p]} b=\lambda^{-(p-1)} c
$$
$$
\eta(a)=\lambda^2,
\eta(b)=\lambda^{p-1}, 
\eta(c)=\lambda^{-2}
$$
Once we have solved $b^{-1}a^{[p]}b = \lambda^{p-1}a$ we may take $c=a^{-1}$. We take
$$
a=a(\lambda^2),
b=b(\lambda^2,\lambda^{p-1}), c=a^{-1}
$$
If we put
$A_1=a\oplus\cdots \oplus a \oplus \ell$,
$B_1 =b\oplus\cdots\oplus b\oplus m$,
$C_1=c\oplus\cdots\oplus c\oplus n$
then
$
A_1, B_1, C_1\in CO_{2n}(q)^\circ, 
$
with
$$
B_1^{-1}A_1^{[p]}B_1 = \lambda^{p-1} A_1
\ ,\ 
C_1^{-1}A_1^\tau C_1 =   A_1
\ , \
B_1^\tau C_1= \lambda^{p-1} C_1^{[p]} B_1
$$
$$
\eta(A_1)=\lambda^2,
\eta(B_1)=\lambda^{p-1},
\eta(C_1)=\lambda^{-2}
$$
so that
$$
\tilde T=\<{A_1,\varphi B_1,\tau C_1}Z/Z
$$
is a $T$-abelian supplement.

 $T=\gen{\delta^2,\phi,\tau\delta}$.
In the $\psl{4}{q}$ case for 
$\gen{\delta^2,\phi,\gamma\delta}$ we took
\[
L=\begin{pmatrix}
	0 & -\lambda & 0 & 0\\
	1 & 0 & 0 & 0\\
	0 & 0 & 0 & -\lambda\\
	0 & 0 & 1 & 0
	\end{pmatrix}, M=\begin{pmatrix}
	\lambda^{\frac{p-1}{2}} & 0 & 0 & 0\\
	0 & 1 & 0 & 0\\
	0 & 0 & \lambda^{\frac{p-1}{2}} & 0\\
	0 & 0 & 0 & 1
	\end{pmatrix}\begin{pmatrix}
	0 & -\lambda & 0 & 0\\
	1 & 0 & 0 & 0\\
	0 & 0 & 1 & 0\\
	0 & 0 & 0 & 1
	\end{pmatrix}^{\frac{1-p}{2}},N=
	\begin{pmatrix}
	0 & -1 & 0 & 0\\
	1 & 0 & 0 & 0\\
	0 & 0 & \lambda^{-1} & 0\\
	0 & 0 & 0 & 1
	\end{pmatrix}
\]
$$
\tilde T=\gen{L,\phi
	M,\gamma
	N}Z/Z
$$
with
$$
M^{-1}L^{[p]}M = z_1 L\ ,\ 
N^{-1}({}^tL^{-1})N = z_2 L\ , \
({}^tM^{-1})N=z_3 N^{[p]} M
$$
$$
z_1=\lambda^{\frac12(p-1)}, z_2= \lambda^{-1}, z_3=1, \det L=\lambda^2, \det M=\lambda^{\frac12(p-1)}, \det N=\lambda^{-1}
$$
Hence in $CO_6(k)^\circ$, with $\ell=\sigma(L)$, $m=\sigma(M)$, $n=\dot w_0\sigma(N)$:
$$
m^{-1}\ell^{[p]}m = \lambda^{p-1} \ell
\ , \ 
n^{-1}\ell^\tau n =   \ell
\ , \ 
m^\tau n= \lambda^{\frac12(p-1)} n^{[p]} m
$$
$$
\eta(\ell)=\det L=\lambda^2,
\eta(m)=\det M=\lambda^{\frac12(p-1)},
\eta(n)=\eta(\dot w_0)\eta(\sigma(N))=\det N=\lambda^{-1}
$$
We have to define matrices $a$, $b$, $c\in CO_4(q)^\circ$ such that
$$
b^{-1}a^{[p]}b = \lambda^{p-1}a
\ , \ 
c^{-1}a c =  a
\ , \ 
b c= \lambda^{\frac12(p-1)} c^{[p]} b\quad,\ \text {i.e.} \quad b^{-1}c^{[p]} b=\lambda^{-\frac12(p-1)} c
$$
$$
\eta(a)=\lambda^2,
\eta(b)=\lambda^{\frac12(p-1)}, 
\eta(c)=\lambda^{-1}
$$
Once we have solved $b^{-1}c^{[p]}b = \lambda^{-\frac12(p-1)}c$ we may take $a=c^{-2}$. We take
$$
c=a(\lambda^{-1}),
b=b(\lambda^{-1},\lambda^{\frac12(p-1)}), a=c^{-2}
$$
If we put
$A_1=a\oplus\cdots \oplus a \oplus \ell$,
$B_1 =b\oplus\cdots\oplus b\oplus m$,
$C_1=c\oplus\cdots\oplus c\oplus n$
then
$
A_1, B_1, C_1\in CO_{2n}(q)^\circ, 
$ with
$$
B_1^{-1}A_1^{[p]}B_1 = \lambda^{p-1} A_1
\ , \ 
C_1^{-1}A_1^\tau C_1 =   A_1
\ , \ 
B_1^\tau C_1= \lambda^{\frac12(p-1)} C_1^{[p]} B_1
$$
$$
\eta(A_1)=\lambda^2,
\eta(B_1)=\lambda^{\frac12(p-1)},
\eta(C_1)=\lambda^{-1}
$$
so that
$$
\tilde T=\<{A_1,\varphi B_1,\tau C_1}Z/Z
$$
is a $T$-abelian supplement.

\subsection{$p\equiv-1 \mod{4}$}
{ {\vskip  0mm plus .5mm } \noindent }

$T=\gen{\delta,\phi\tau}$. In the $\psl{4}{q}$ case for
$\gen{\delta,\phi\gamma}$ we took

	\[
		L=\begin{pmatrix}
		0 & 0 & 0 & -\lambda\\
		1 & 0 & 0 & 0\\
		0 & 1 & 0 & 0\\
		0 & 0 & 1 & 0
		\end{pmatrix}, M=
		\begin{pmatrix}
		\lambda^{\frac{3(-p-1)}{4}} & 0 & 0 & 0\\
		0 & \lambda^{\frac{2(-p-1)}{4}} & 0 & 0\\
		0 & 0 & \lambda^{\frac{-p-1}{4}} & 0\\
		0 & 0 & 0 & 1
		\end{pmatrix}
	\]
$$
\tilde T=\<{L,\varphi \gamma M}Z/Z
$$
with
$$
M^{-1}({}^t(L^{[p]})^{-1}) M=\lambda^{-\frac{p+1}{4}}L
\quad,\quad
\det L=\lambda, \det M=\lambda^{-\frac32(p+1)}
$$
Hence in $CO_6(k)^\circ$, with $\ell=\sigma(L)$, $m=\dot w_0\sigma(M)$:
$$
m^{-1}(\ell^{[p]})^\tau m = \lambda^{\frac12(p-1)} \ell
$$
$$
\eta(\ell)=\det L=\lambda,
\eta(m)=\eta(\dot w_0)\eta(\sigma(M))=\det M=\lambda^{-\frac32(p+1)}
$$
We have to define matrices $a$, $b\in CO_4(q)^\circ$ such that
$$
b^{-1}a^{[p]}b = \lambda^{\frac12(p-1)}a
\quad,\quad
\eta(a)=\lambda,
\eta(b)=\lambda^{-\frac32(p+1)}
$$
We take
$$
a=a(\lambda),
b=b(\lambda,\lambda^{-\frac32(p+1)})
$$
If we put
$A_1=a\oplus\cdots \oplus a \oplus \ell$,
$B_1 =b\oplus\cdots\oplus b\oplus m$
then
$
A_1, B_1\in CO_{2n}(q)^\circ, 
$ with
$$
B_1^{-1}(A_1^{[p]})^\tau B_1 = \lambda^{\frac12(p-1)} A_1
\quad,\quad
\eta(A_1)=\lambda,
\eta(B_1)=\lambda^{-\frac32(p+1)}
$$
so that
$$
\tilde T=\<{A_1,\varphi \tau B_1}Z/Z
$$
is a $T$-abelian supplement.

 $T=\gen{\delta^2,\phi,\tau}$.
In the $\psl{4}{q}$ case for
$\gen{\delta^2,\phi, \gamma}$, we took
$$
\tilde T= \<{L,\varphi M, \gamma N}Z/Z
$$
with the same $L$, $M$, $N$ as in the case $p\equiv 1\mod 4$, $T=\gen{\delta^2,\phi, \gamma}$. We define $A_1, B_1, C_1\in CO_{2n}(q)^\circ$ as in this case and $\tilde T=\<{A_1,\varphi B_1,\tau C_1}Z/Z$
is a $T$-abelian supplement.

 $T=\gen{\delta^2,\phi\delta,\tau\delta}$.
In the $\psl{4}{q}$ case for
$\gen{\delta^2,\phi\delta, \gamma\delta}$, we took
$$
\tilde T= \<{L,\varphi M, \gamma N}Z/Z
$$
with the same $L$, $M$, $N$ as in the case $p\equiv 1\mod 4$, $T=\gen{\delta^2,\phi, \gamma\delta}$. Again, we define $A_1, B_1, C_1\in CO_{2n}(q)^\circ$ in the same way and $\tilde T=\<{A_1,\varphi B_1,\tau C_1}Z/Z$
is a $T$-abelian supplement.

We have proved
\begin{thm}\label{pr_Dnodd}
	Let $G$ be an almost simple group with socle $G_0=D_n(q)$, $n$ odd. If $G/G_0$ is abelian, then there exists an abelian subgroup $A$ such that $G=AG_0$.
\end{thm}

\section{$D_n(q)$, $n$ even}\label{sec_Deven}	

Here $L$ is of type $D_n$, $n$ even, $G_0=D_n(q)$, $q=p^m$,  $d=(2 ,q-1)^2$ and we assume that $\aut(G_0)$ does not split over $G_0$, hence $( \frac{q^n-1}d,d,m)\not=1$. In particular, $d\not=1$, hence $p$ is odd, $m$ is even and $d=4$, $\hat H/H\cong C_2\times C_2$. 

If $n=4$ 
$$
\out(G_0)=(\<{\delta_1, \delta_2, \delta_3}\times \<\varphi) :\ S_3
$$
$S_3=\<{\rho, \tau}$, $\tau^2=1$, $\rho^3=1$, $\delta_1\delta_2=\delta_3$, $\delta_i^2=\varphi^m=[\rho, \varphi]=[\tau,\varphi]=1$, $\delta_1^\tau=\delta_2$, $\delta_3^\tau=\delta_3$, $
\delta_1^\rho=\delta_2$, $\delta_2^\rho=\delta_3$, $\delta_3^\rho=\delta_1$.

If $n\not=4$
$$
\out(G_0)=(\<{\delta_1, \delta_2, \delta_3}\times \<\varphi) :\<\tau
$$
$\tau^2=1$, $\delta_1\delta_2=\delta_3$, $\delta_i^2=\varphi^m=[\tau,\varphi]=1$, $\delta_1^\tau=\delta_2$, $\delta_3^\tau=\delta_3$.

Note that $(\tau\delta_1)^2=\tau\delta_1\tau\delta_1=\delta_2\delta_1=\delta_3$,  $(\tau\delta_2)^2=\delta_3$, hence
$$
\<{\delta_3,\varphi,\tau\delta_1}=\<{\varphi,\tau\delta_1}=\<{\varphi,\tau\delta_2}
$$
We have to consider the following cases. Assume $T$ is an abelian, non-cyclic subgroup of $ \<{\delta_1,\delta_2,\varphi,\tau}$ (which is $\out(G_0)$ if $n\not=4$).

Let $D=\<{\delta_1,\delta_2}$, $\pi:\out(G_0)\to\out(G_0)/D=\<{\varphi,\tau}$. If $\pi(T)$ is cyclic, then
$
\pi(T)=\<{\varphi^s\tau^\epsilon}
$.
If $\varepsilon=0$, $T\leq\<{D,\varphi^s}$, so
$$
T\leq \<{\delta_1,\delta_2,\varphi}
$$
If $\varepsilon=1$, $T\leq\<{D,\varphi^s\tau}$, and $T$ contains an element $\alpha=\varphi^s\tau \delta$, $\delta\in D$, $\delta\not=1$, so 
either
$T=\<{\delta_3,\varphi^s\tau}$ or $T=\<{\delta_3,\varphi^s\tau\delta_1}=\<{\delta_3,\varphi^s\tau\delta_2}$. In the first case
$$
T\leq \<{\delta_3,\varphi,\tau}
$$
In the second case
$$
T\leq \<{\varphi,\tau\delta_1}=\<{\varphi,\tau\delta_2}
$$

If $\pi(T)$ is not cyclic, then $\pi(T)=\<{\varphi^s,\tau}$. Therefore either
$$
T\leq \<{\delta_3,\varphi,\tau}
$$
or 
$$
T\leq \<{\delta_3,\varphi,\tau\delta_1}=\<{\varphi,\tau\delta_1}=\<{\varphi,\tau\delta_2}
$$
Therefore if $T\leq \<{\delta_1,\delta_2,\varphi,\tau}$, we only have to deal with cases:
\begin{equation*}
\begin{split}
\text{case 1:}&\quad T= \<{\delta_1,\delta_2,\varphi}\\
\text{case 2:}&\quad T= \<{\delta_3,\varphi,\tau}
\\
\text{case 3:}&\quad T= \<{\varphi,\tau\delta_1}=\<{\varphi,\tau\delta_2}
\end{split}
\end{equation*}
\vskip10pt
Assume $n=4$.
Let $M=\<{\delta_1,\delta_2,\varphi}$, $\zeta:\out(G_0)\to\out(G_0)/M=\<{\rho,\tau}$, $T$ an abelian, not cyclic subgroup of $\out(G_0)$, $T$ not contained in $\<{\varphi,\rho,\tau}$ (otherwise we are done, by taking $\tilde T=T$). Hence $\zeta(T)=\{1\}$,  $\<{\rho^i\tau}$ or $\<\rho$. However $\rho^i\tau$ is conjugate to $\tau$, therefore we may assume
$\zeta(T)=\{1\}$, $\<{\tau}$ or $\<\rho$.

If $\zeta(T)=\{1\}$, $\<\tau$ we are in the previous case $T\leq \<{\delta_1,\delta_2,\varphi,\tau}$. We are left with $\zeta(T)=\<\rho$, $T\leq \<{\delta_1,\delta_2,\varphi,\rho}$, so $T=\<{\varphi^s,\varphi^t\rho\delta}$, $\delta\in D$, $\delta\not=1$ since $T$ is abelian and not contained in $\<{\varphi,\rho,\tau}$. It follows that $T\leq \<{\varphi,\rho\delta}$. Moreover, since $\<\rho$ acts transitively on $\{\delta_1,\delta_2,\delta_3\}$ and $[\rho,\varphi]=1$, we may assume
$$
\text{case 4:}\quad T= \<{\varphi,\rho\delta_2}\quad\quad\text{only for $D_4$}
$$
We use the same procedure used to deal with the odd $n$ case. It is convenient to start with $G_0=D_2(q)=P\Omega_4^+(q)\cong \psl{2}{q} \times \psl{2}{q}$.

\vskip10pt
We have
$n_1=\left(
\begin{smallmatrix}
 0 & 1 & 0 & 0 \\
 -1 & 0 & 0 & 0 \\
 0 & 0 & 0 & 1 \\
 0 & 0 & -1 & 0 \\
\end{smallmatrix}
\right)
$,
$n_2=
\left(
\begin{smallmatrix}
 0 & 0 & 0 & 1 \\
 0 & 0 & -1 & 0 \\
 0 & 1 & 0 & 0 \\
 -1 & 0 & 0 & 0 \\
\end{smallmatrix}
\right)
$
in $\Omega^+_4(q)$. Note that $n_2=\tau_2n_1\tau_2$ and $n_1n_2=n_2n_1$.
If
$
g=
\left(
\begin{smallmatrix}
f_1& 0 & 0 & 0 \\
 0 & f_2 & 0 & 0 \\
 0 & 0 & \frac{\mu} {f_1}& 0 \\
 0 & 0 & 0 & \frac{\mu}{f_2} \\
\end{smallmatrix}
\right)
$
is a diagonal matrix in $CO_4(q)^\circ$
then $\alpha_1(g)=\frac{f_1}{f_2}$, $\alpha_2(g)=\frac{f_1f_2}{\mu}$. We define $\delta_1$, $\d_2$, $\d_3$.
Let
$
h_1=
\left(
\begin{smallmatrix}
 1 & 0 & 0 & 0 \\
 0 & \lambda & 0 & 0 \\
 0 & 0 & \lambda & 0 \\
 0 & 0 & 0 & 1 \\
\end{smallmatrix}
\right)
$
in $CO_4(q)^\circ$.
Then $\a_1(h_1)=\lambda^{-1}$, $\a_2(h_1)=1$. We write for short $h_1\mapsto h(\chi_1)\in \hat H$ where $\chi_1=(\lambda^{-1},1)$ is the $k$-character of $Q$ with $\chi_1(\a_1)=\lambda^{-1}$, $\chi_1(\a_2)=1$. We define $\delta_1:=h(\chi_1)G_0$. Moreover $\chi_2:=\chi_1\circ \tau=(1,\lambda^{-1})$, $h_2:=h_1^\tau$, $h_2\mapsto h(\chi_2)\in \hat H$, 
$\delta_2:=h(\chi_2)G_0$; finally $h_3:=h_1h_2$, 
hence
$
h_3\mapsto h(\chi_3)$,
$\chi_3=\chi_1+\chi_2=(\lambda^{-1},\lambda^{-1})$,
$\delta_3:=h(\chi_3)G_0$,
so 
$\delta_3=\delta_1\delta_2$.

Let $
x_1=n_1h_1=
\left(
\begin{smallmatrix}
 0 & \lambda & 0 & 0 \\
 -1 & 0 & 0 & 0 \\
 0 & 0 & 0 & 1 \\
 0 & 0 & -\lambda & 0 \\
\end{smallmatrix}
\right)
$, $
y=
\left(
\begin{smallmatrix}
 \lambda^{p-1} & 0 & 0 & 0 \\
 0 & \lambda^{\frac{p-1}{2}} & 0 & 0 \\
 0 & 0 & 1 & 0 \\
 0 & 0 & 0 & \lambda^{\frac{p-1}{2}} \\
\end{smallmatrix}
\right)$. Then $x_1$, $y$ are in $CO_4(q)^\circ$ with
$$
\eta(x_1)=\eta(h_1)=\lambda\ , \
\eta(y)=\lambda^{p-1}
\ ,\ 
y^{-1}x_1^{[p]} y=\lambda^{\frac{p-1}{2}} x_1\ , \ 
y^\tau=y
$$

We take $x_2:= x_1^\tau=n_2h_2$, Then 
$$
\eta(x_2)=\eta(h_2)=\lambda
\ , \
y^{-1}x_2^{[p]} y=\lambda^{\frac{p-1}{2}} x_2
\ , \
x_1x_2=x_2x_1
$$
We put $x_3:=x_1x_2$. We have $x_3^\tau=x_3$, $\eta(x_3)=\lambda^2$.
 Since $n_1$, $n_2$ are in $\Omega_4^+(q)$, $x_i$ induces $\delta_i$ for $i=1,2,3$. On the other hand, we have
$
y\mapsto h(\chi)$, 
$\chi=(\lambda^{\frac{p-1}{2}},\lambda^{\frac{p-1}{2}})$, 
so
$y$ induces 
$\delta_3$
if $p\equiv -1 \mod 4$ and
the identity if $p\equiv 1\mod 4$.

We are in a position to deal with the 3 cases for $D_2$:

\noindent
Case 1: $T=\<{\d_1, \d_2}\times \<\varphi$.
From the above we have
$$
x_1x_2=x_2x_1\quad,\quad
y^{-1}x_1^{[p]}y=\lambda^{\frac{p-1}{2}}x_1\quad,\quad
y^{-1}x_2^{[p]}y=\lambda^{\frac{p-1}{2}}x_2
$$
hence
$$
\tilde T=\<{x_1,x_2,\varphi y}Z/Z
$$
is a $T$-abelian supplement.

\noindent
Case 2: $T=\<{\d_3}\times \<\varphi\times \<\tau$.
We have 
$$
x_3^\tau=x_3\quad,\quad y^\tau=y\quad,\quad
y^{-1}x_3^{[p]}y=\lambda^{p-1}x_3\quad,\quad
$$
hence
$$
\tilde T=\<{x_3,\varphi y,\tau}Z/Z
$$
is a $T$-abelian supplement.

\noindent
Case 3: $T=\<{\tau\delta_1,\varphi}$.
We have
$$
y^\tau=y\quad,\quad
y^{-1}x_1^{[p]}y=\lambda^{\frac12(p-1)}x_1
$$
hence
$$
\tilde T=\<{\tau x_1,\varphi y}Z/Z
$$
is a $T$-abelian supplement. Note that $y$ induces the identity if $p\equiv 1\mod 4$ and $\delta_3$ if $p\equiv -1\mod 4$, but $x_3=(\tau x_1)^2$.

We now deal with $G_0=D_n(q)$, $n$ even, $n=2m$, $n\geq 4$.
Let $
c_i=\alpha_i, i=1,\ldots,n-2
$,
$$
c_{n-1}=\alpha_{n-1}-(\alpha_1+\alpha_3+\cdots+\alpha_{n-3})\quad,\quad
c_{n}=\alpha_{n}-(\alpha_1+\alpha_3+\cdots+\alpha_{n-3})
$$
Then
$
\frac12c_{n-1}$, $\frac12c_n$ are in $P
$
hence $(c_1,\ldots,c_n)$ is a $\mathbb Z$-basis of $Q$ and $(c_1,\ldots,c_{n-2},\frac12c_{n-1},\frac12c_n)$ is a $\mathbb Z$-basis of $P$. If $\chi:Q\to k^\times$ is a character, then $\chi$ can be extended to a character of $P$ if and only if $\chi(c_{n-1})$ and $\chi(c_n)$ are in $(k^\times)^2$.

We define the characters $\psi_1$, $\psi_2$, $\psi_3:Q\to k^\times$. As usual, $\lambda$ is a generator of $k^\times$.
$$
\psi_1(\alpha_i)=1,\  i=1,\ldots,n-2\ ,\  \psi_1(\alpha_{n-1})=\lambda\ , \ \psi_1(\alpha_{n})=1
$$
hence
$\psi_1(c_{n-1})=\lambda$, $\psi_1(c_{n})=1$. Then we put $\psi_2=\psi_1\circ\tau$, so $\psi_2(c_{n-1})=1$, $\psi_2(c_{n})=\lambda$, and $\psi_3=\psi_1+\psi_2$, so $\psi_3(c_{n-1})=\psi_3(c_{n})=\lambda$. Finally
$\delta_1:=h(\psi_1)G_0$,
$\delta_2:=h(\psi_2)G_0$,
$\delta_3:=h(\psi_3)G_0$, hence $\delta_3=\delta_1\delta_2$. Each $\d_i$ induces the corresponding diagonal automorphism of $D_2(q)$ relative to $\alpha_{n-1}$, $\alpha_n$ (denoted above with the same symbols). 

Let $U=\<{e_1,\ldots,e_{n-2},f_1,\ldots,f_{n-2}}$. Then 
$k^{2n}$ is the orthogonal direct sum
$k^{2n}=U\oplus U^\perp$, $U^\perp=\<{e_{n-1},e_n,f_{n-1},f_n}$, with $\dim U=2n-4=4(m-1)$.
Moreover $U$ is the direct orthogonal sum of subspaces of dimension 4:
$$
U_1=\<{e_1,e_2,f_1,f_2},\ldots, U_{m-1}=\<{e_{n-3}, e_{n-2},f_{n-3},f_{n-2}}
$$
To define an isometry or more generally an orthogonal similitude of $k^{2n}$ we give matrices 
$X_i\in CO_4(q)^\circ$, $\eta(X_i)=\mu$, $i=1,\ldots,m-1$, $X\in CO_4(q)^\circ$, $\eta(X)=\mu$ and define $Y$ in ${\rm GL}_{2n}(q)$ by
$$
Y=X_1\oplus \cdots \oplus X_{m-1}\oplus X
$$
Then $Y\in CO_{2n}(q)^\circ$, with $\eta(Y)=\mu$. If $Y\in CO_{2n}(q)^\circ$ fixes $U^\perp$, then it fixes $U$ and if we write $Y=X\oplus Z$, with $X\in CO_4(q)^\circ$, $Z\in CO_{2n-4}(q)^\circ$, and consider the action of $\varphi$ and $\tau$, we get
$$
Y^{[p]}=X^{[p]}\oplus Z^{[p]}\quad,\quad
Y^\tau=X^\tau \oplus Z
$$
since $\tau$ acts on the basis $(e_1,\ldots,e_n,f_1,\ldots,f_n)$ just switching $e_n$ and $f_n$ (here $Y^\tau =\tau_n Y\tau_n$,
$X^\tau = \tau_2 X\tau_2$). 

We shall proceed as follows. Assume $T$ is an abelian subgroup of $\Out (G_0)$, $G_0$ of type $D_n$ ($T\leq \<{\delta_1,\delta_2,\varphi,\tau}$ if $G_0=D_4(q)$). We consider the analogous subgroup $T$ of $\Out (D_2(q))$. From the $D_2$ case, we have an abelian subgroup of $\aut (D_2(q))$ given by explicit matrices in $CO_4(q)^\circ$. For each such matrix $X$ we define a matrix $X_1\in CO_4(q)^\circ$ and finally define a matrix $Y=X_1\oplus\cdots\oplus X_1\oplus X$ in $CO_{2n}(q)^\circ$ ($m-1$ copies of $X_1$). We shall then obtain a $T$-abelian supplement $\tilde T$ in $\Aut (G_0)$.

\vskip 20pt
Recall the matrices $a(\mu)$, $b(\mu,\nu)$ in $CO_4(q)^\circ$ defined in (\ref{A}), (\ref{B}) and the matrices $x_1$, $x_2$, $x_3$, $y\in CO_4(q)^\circ$ defined to deal with $D_2$. We had
$$
x_1x_2=x_2x_1\quad,\quad
y^{-1}x_1^{[p]}y=\lambda^{\frac{p-1}{2}}x_1\quad,\quad
y^{-1}x_2^{[p]}y=\lambda^{\frac{p-1}{2}}x_2
$$
$\eta(x_1)=\eta(x_2)=\lambda$, $\eta(y)=\lambda^{p-1}$.
$$
x_3^\tau=x_3\quad,\quad y^\tau=y\quad,\quad
y^{-1}x_3^{[p]}y=\lambda^{p-1}x_3
\quad,\quad
\eta(x_3)=\lambda^2
$$
We take $\mu=\lambda$,  $\nu=\lambda^{p-1}$, i.e.
$$
a=a(\lambda)=
\left(
\begin{array}{cccc}
 0 & \lambda& 0 & 0 \\
 -1 & 0 & 0 & 0 \\
 0 & 0 & 0 & 1 \\
 0 & 0 & -\lambda & 0 \\
\end{array}
\right) (=x_1)\ 
, \ \eta(a)=\lambda
$$
$$
b=b(\lambda,\lambda^{p-1})
=\left(
\begin{array}{cccc}
  \lambda^{\frac{1}{2}(p-1)} & 0 & 0 & 0 \\
 0 & 1 & 0 & 0 \\
 0 & 0 &  \lambda^{\frac{1}{2}(p-1)} & 0 \\
 0 & 0 & 0 & \lambda^{p-1} \\
\end{array}
\right)
\ , \
\eta(b)=\lambda^{p-1}
$$
so
$
b^{-1}a^{[p]}b=\lambda^{\frac12(p-1)} a$.
We put
$A_1=\underbrace{a\oplus\cdots \oplus a}_{m-1} \oplus x_1$,
$A_2=A_1^\tau=\underbrace{a\oplus\cdots \oplus a}_{m-1} \oplus x_2$,
$B =\underbrace{b\oplus\cdots\oplus b}_{m-1}\oplus y$.
Then $A_1, A_2, B\in CO_{2n}(q)^\circ$, $\eta(A_1)=\eta(A_2)=\lambda$, $\eta(B)=\lambda^{p-1}$ and
$$
A_1A_2=A_2A_1\quad,\quad
B^{-1}A_1^{[p]}B=\lambda^{\frac{1}{2}(p-1)} A_1\quad,\quad
B^{-1}A_2^{[p]}B=\lambda^{\frac{1}{2}(p-1)} A_2
$$
If moreover $A_3=A_1A_2$, then
$$
A_3^\tau=A_3\quad,\quad
B^\tau=B\quad,\quad
B^{-1}A_3^{[p]}B=\lambda^{p-1}A_3
$$
$\eta(A_3)=\lambda^2$. Moreover
$(\tau A_1)^2=A_1^\tau A_1=A_2A_1=A_3$.

We have (recall that $n=2m$) for $\gamma=\lambda^{\frac12(p-1)}$
$$
(\alpha_i(B))_{i=1,\ldots,n}=(\underbrace{\gamma,\gamma^{-1},\ldots,\gamma,\gamma^{-1}}_{n-4},\gamma,\gamma^{-2},\gamma,\gamma)
$$
$$
c_{n-1}(B)=c_n(B)=\gamma^{2-m}=\lambda^{\frac12(p-1)(2-m)}
$$
so $B$ induces $\delta_3$ if $m$ is odd and $p\equiv -1\mod 4$, and the identity otherwise. 

For $\mu\in k^\times$ let 
$$
h(\mu)=
\left(
\begin{array}{cccc}
 1 & 0 & 0 & 0 \\
0 & \mu & 0 & 0 \\
 0 & 0 & 1 & 0\\
 0 & 0 & 0 & \mu^{-1}  \\
\end{array}
\right)
\left(
\begin{array}{cccc}
 1 & 0 & 0 & 0 \\
 0 & 1 & 0 & 0 \\
 0 & 0 & \mu & 0 \\
 0 & 0 & 0 & \mu \\
\end{array}
\right)=
\left(
\begin{array}{cccc}
 1 & 0 & 0 & 0 \\
 0 & \mu & 0 & 0 \\
 0 & 0 & \mu & 0 \\
 0 & 0 & 0 & 1 \\
\end{array}
\right)
, \eta(h(\mu))=\mu
$$

If $H(\mu)=\underbrace{h(\mu)\oplus\cdots \oplus h(\mu)}_m $ in $CO_{2n}(q)^\circ$, then
$$
(\alpha_i(H(\mu)))_{i=1,\ldots,n}=(\underbrace{\mu^{-1},\mu,\ldots,\mu^{-1},\mu}_{n-2},\mu^{-1},1)
$$
$$
c_{n-1}(H(\mu))=\mu^{m-2}
\quad,\quad
c_n(H(\mu))=\mu^{m-1}
$$
Note that $A_1$ induces the same diagonal automorphism in $\out(G_0)$ as $H(\lambda)$ since $A_1H(\lambda)^{-1}\in N$.
Therefore $A_1$ induces $\delta_1$ if $m$ is odd, $\delta_2$ if $m$ is even. Hence $A_2$ induces $\delta_2$ if $m$ is odd, $\delta_1$ if $m$ is even. It follows that $A_3$ induces $\delta_3$.

\noindent
Case 1: $T=\<{\d_1, \d_2}\times \<\varphi$.  In the $D_2(q)$ case
we took
$
\tilde T=\<{x_1,x_2,\varphi y}Z(CO_{4}(k)^\circ)/Z(CO_{4}(k)^\circ)
$.
Then
$$
\tilde T=\<{A_1,A_2,\varphi B}Z(CO_{2n}(k)^\circ)/Z(CO_{2n}(k)^\circ)
$$
is a $T$-abelian supplement in $\aut(G_0)$.

\noindent
Case 2: $T=\<{\d_3}\times \<\varphi\times \<\tau$.  In the $D_2(q)$ case
we took
$
\tilde T=\<{x_3,\varphi y,\tau}Z/Z
$.
Then
$$
\tilde T=\<{A_3,\varphi B,\tau}Z/Z
$$
is a $T$-abelian supplement in $\aut(G_0)$.

\noindent
Case 3: $T=\<{\tau\delta_1,\varphi}$.  In the $D_2(q)$ case
we took
$
\tilde T=\<{\tau x_1,\varphi y}Z/Z
$.
Then
$$
\tilde T=\<{\tau A_1,\varphi B}Z/Z
$$
is a $T$-abelian supplement in $\aut(G_0)$.

We finally deal with the last case

\noindent
Case 4: $T=\<{\varphi,\rho\delta_2}$, only for $D_4(q)$. We have defined the matrices $A_1$, $B$ in $CO_{2n}(q)^\circ$: in the case $n=4$ they are
$$
A_1=
\left(
\begin{smallmatrix}
 0 & \lambda & 0 & 0 & 0 & 0 & 0 & 0 \\
 -1 & 0 & 0 & 0 & 0 & 0 & 0 & 0 \\
 0 & 0 & 0 & \lambda & 0 & 0 & 0 & 0 \\
 0 & 0 & -1 & 0 & 0 & 0 & 0 & 0 \\
 0 & 0 & 0 & 0 & 0 & 1 & 0 & 0 \\
 0 & 0 & 0 & 0 & -\lambda & 0 & 0 & 0 \\
 0 & 0 & 0 & 0 & 0 & 0 & 0 & 1 \\
 0 & 0 & 0 & 0 & 0 & 0 & -\lambda & 0 \\
\end{smallmatrix}
\right)\ , \
B=
\left(
\begin{smallmatrix}
 \lambda^{\frac{p-1}{2}} & 0 & 0 & 0 & 0 & 0 & 0 & 0 \\
 0 & 1 & 0 & 0 & 0 & 0 & 0 & 0 \\
 0 & 0 & \lambda^{p-1} & 0 & 0 & 0 & 0 & 0 \\
 0 & 0 & 0 & \lambda^{\frac{p-1}{2}} & 0 & 0 & 0 & 0 \\
 0 & 0 & 0 & 0 & \lambda^{\frac{p-1}{2}} & 0 & 0 & 0 \\
 0 & 0 & 0 & 0 & 0 & \lambda^{p-1} & 0 & 0 \\
 0 & 0 & 0 & 0 & 0 & 0 & 1 & 0 \\
 0 & 0 & 0 & 0 & 0 & 0 & 0 & \lambda^{\frac{p-1}{2}} \\
\end{smallmatrix}
\right)
$$
We have
$$
A_1=n_1n_3
\left(
\begin{smallmatrix}
 1 & 0 & 0 & 0 & 0 & 0 & 0 & 0 \\
 0 & \lambda & 0 & 0 & 0 & 0 & 0 & 0 \\
 0 & 0 & 1 & 0 & 0 & 0 & 0 & 0 \\
 0 & 0 & 0 & \lambda & 0 & 0 & 0 & 0 \\
 0 & 0 & 0 & 0 & 1 & 0 & 0 & 0 \\
 0 & 0 & 0 & 0 & 0 & \frac{1}{\lambda} & 0 & 0 \\
 0 & 0 & 0 & 0 & 0 & 0 & 1 & 0 \\
 0 & 0 & 0 & 0 & 0 & 0 & 0 & \frac{1}{\lambda} \\
\end{smallmatrix}
\right)
\left(
\begin{smallmatrix}
 1 & 0 & 0 & 0 & 0 & 0 & 0 & 0 \\
 0 & 1 & 0 & 0 & 0 & 0 & 0 & 0 \\
 0 & 0 & 1 & 0 & 0 & 0 & 0 & 0 \\
 0 & 0 & 0 & 1 & 0 & 0 & 0 & 0 \\
 0 & 0 & 0 & 0 & \lambda & 0 & 0 & 0 \\
 0 & 0 & 0 & 0 & 0 & \lambda & 0 & 0 \\
 0 & 0 & 0 & 0 & 0 & 0 & \lambda & 0 \\
 0 & 0 & 0 & 0 & 0 & 0 & 0 & \lambda \\
\end{smallmatrix}
\right)=n_1n_3 H(\lambda)
$$
In $P(CO_8(q)^\circ)=G_0\hat H$ we obtain the elements
$$
A_1\mapsto
n_1n_3 h(\xi_1)\in G_0\hat H\ , \ 
B\mapsto h(\xi)\in \hat H
$$
where $
\xi_1
$ is the $\F_q$-character of $Q$
$$
\a_1\mapsto \lambda^{-1},
\a_2\mapsto \lambda,
\a_3\mapsto \lambda^{-1},
\a_4\mapsto 1\quad
$$
In particular
$
c_3\mapsto 1$,
$c_4\mapsto \lambda
$
so that $n_1n_3 h(\xi_1)$ induces $\delta_2$ in $\out{G_0}$, while $\xi$ is the $\F_q$-character of $Q$
$$
\a_1\mapsto \lambda^{\frac{p-1}{2}},
\a_2\mapsto \lambda^{1-p},
\a_3\mapsto \lambda^{\frac{p-1}{2}},
\a_4\mapsto \lambda^{\frac{p-1}{2}}
$$
In particular
$c_3\mapsto 1$,
$c_4\mapsto 1
$, so $\xi$
can be extended to a character of $P$, hence $h(\xi)\in H$.
From
$
B^{-1}A_1^{[p]}B=\lambda^{\frac{1}{2}(p-1)} A_1
$ 
we get
$
[\varphi h(\xi),n_1n_3 h(\xi_1)]=1
$.
Moreover $h(\xi)^\rho=h(\xi)$, hence
$$
\tilde T=\<{\varphi h(\xi),\rho n_1n_3 h(\xi_1)} 
$$
is a $T$-abelian supplement in $\aut(G_0)$.

We have proved
\begin{thm}\label{pr_Dneven}
	Let $G$ be an almost simple group with socle $G_0=D_n(q)$, $n$ even. If $G/G_0$ is abelian, then there exists an abelian subgroup $A$ such that $G=AG_0$.
\end{thm}

This completes the proof of Theorem \ref{almost}.

\section{Proof of Corollary \ref{fittrivial}}\label{sec_corollary}

In the following we will denote by $F(G)$ and $F^*(G)$ respectively, the Fitting subgroup and the generalized Fitting subgroup of $G.$

\begin{proof}[Proof of Corollary \ref{fittrivial}]Notice that $F(G)=1$ implies $N=\soc(G)=F^*(G).$ Let
	$H=\langle a, b, N\rangle.$ If $M$ is a minimal normal subgroup of
	$H,$ then either $M \leq N$ or $M\cap N=1.$ However in the second case we would have $M\leq C_G(N)=C_G(F^*(G))=Z(F^*(G))=1,$ a contradiction. This implies $N=\soc(H)=F^*(H),$ and therefore it is not restrictive to assume $G=\langle a, b, N\rangle.$ 
	
	We decompose $N=N_1\times \dots \times N_u$ as a product of minimal normal subgroups of $G$ and for $1\leq i\leq u$ we denote by $\rho_i: G\to \aut(N_i)$ the map induced by the conjugation action of $G$ on $N_i.$ The map $\rho: G \to \prod_{1\leq i \leq u}\aut(N_i)$ which sends $g$ to
	$(g^{\rho_1},\dots,g^{\rho_u})$	is an injective homomorphism, since
	$\ker \rho= \bigcap_{1\leq i\leq u}C_G(N_i)=C_G(N)=1.$
	If $u\neq 1,$ then by induction there exist $n_i, m_i \in N_i$ such that $[(an_i)^{\rho_i},(bm_i)^{\rho_i}]=1.$ But then, setting
	$n=(n_1,\dots,n_u)$ and $m=(m_1,\dots,m_u),$ we have that
	$[(an)^\rho,(bm)^\rho]=1,$ and consequently, since $\rho$ is injective, $[an,bm]=1.$

	Hence it is not restrictive to assume that $N$ is a minimal normal subgroup of $G=\langle a, b, N\rangle.$ Write $N=S_1\times \dots \times S_t,$ where $S_1,\dots,S_t$ are isomorphic non-abelian simple groups, and let
	$X=N_G(S_1)/C_G(S_1).$ We may identify $G$ with a subgroup of $X \wr \perm(t),$
	so $a=x\sigma,$ $b=y\tau,$ with $x,y \in X^t$ and $\langle \sigma, \tau \rangle$ is an abelian regular subgroup of $\perm(t).$ Notice that
	$$\frac{X}{S_1} \cong \frac{N_G(S_1)/C_G(S_1)}{S_1C_G(S_1)/C_G(S_1)}\cong \frac{N_G(S_1)}{S_1C_G(S_1)}.
	$$
	Since $S_1C_G(S_1)\geq N,$ it follows that $X/S_1$ is isomorphic to a section of $G/N.$ Since $G/N$ is an abelian group, $X/S_1$ is abelian and therefore by Theorem \ref{almost} there exists an abelian subgroup $Y$ of $X$ such that
	$X=YS_1$. Then it is not restrictive to assume $\langle a, b \rangle \leq Y \wr \langle \sigma ,\tau \rangle.$ Let $K=\langle a, b\rangle$ and $Z=Y\cap S_1.$ The group $KZ^t/Z^t$ is abelian and we have reduced our problem to find 
	$n, m \in Z^t$ such that  $\langle xn\sigma, ym\tau \rangle$ is abelian. We have
	\begin{align*}
	[xn\sigma,ym\tau]&=[xn\sigma,\tau][xn\sigma,ym]^\tau=[xn,\tau]^\sigma[\sigma,\tau][xn,ym]^{\sigma\tau}[\sigma,ym]^\tau\\&=[xn,\tau]^\sigma [\sigma,ym]^\tau
	=[x,\tau]^\sigma[\sigma,y]^\tau[n,\tau]^\sigma[\sigma,m]^\tau.
	\end{align*}
	Since $[n,\tau]^\sigma  [\sigma,m]^\tau=[n^\sigma,\tau][\sigma,m^\tau],$
	we are looking for $n, m\in Z^t$ such that
	$$[x,\tau]^\sigma[\sigma,y]^\tau=[x\sigma,y\tau]=[\tau,n^\sigma][m^\tau, \sigma].$$
	Notice that $[x\sigma,y\tau]=(z_1,\dots,z_t)\in Z^t,$ with $z_1z_2\cdots z_t=1.$ Let $$\Lambda:=\{(z_1,\dots,z_t)\in Z^t\mid 
	z_1z_2\cdots z_t=1\}.$$
	In order to conclude our proof, it suffices to prove that for every $(z_1,\dots,z_t)\in \Lambda$  there exist $\tilde n, \tilde m \in Z^t$ such that $(z_1,\dots,z_t)=[\tau,\tilde n][\tilde m, \sigma].$

	Since $\langle \sigma, \tau\rangle$ is a regular subgroup of $\perm(t),$ $\sigma=\sigma_1\cdots\sigma_r$ is the product of $r$ disjoint cycles of the same length $s$, with $rs=t.$ First assume $r=1.$ In that case for every $\lambda \in \Lambda,$ there exists $\tilde m \in Z^t$ such that $[\tilde m,\sigma]=\lambda,$ and our conclusion follows taking $\tilde n=1.$
	Finally, assume $r\neq 1.$ In this case $\tau=\tau_1\cdots \tau_u$ is the product of $u$ disjoint cycles of the same length  and $\tau$ must permute cyclically the orbits $\Sigma_1, \dots, \Sigma_r$ of $\sigma.$ It is not restrictive to assume that $i \in \Sigma_i$ for $1\leq i \leq r$ and that $\tau_1(j)=j+1$ for $1\leq j\leq r-1.$ Notice that $[Z^t,\sigma]$ consists of the elements $(k_1,\dots,k_t)\in Z^t$ with the property that, for any $1\leq i\leq r$, $\prod_{\omega\in \Sigma_i}k_\omega=1.$ Given $\lambda \in \Lambda,$
	we may choose $\tilde m$ so that
	$\lambda[\tilde m, \sigma]^{-1}=(v_1,\dots,v_t)\in Z^t$ with
	$v_1\cdots v_r=1$ and $v_j=1$ if $j>r.$ But then we may find $\tilde n=(w_1,\dots,w_r,1,\dots,1)$ so that $[\tau,\tilde n]=[\tau_1,\tilde n]=(v_1,\dots,v_t)$ and therefore $\lambda=[\tau,\tilde n][\tilde m,\sigma].$
\end{proof}


\begin{thebibliography}{99}

\bibitem{bourbaki}{\sc N.\ Bourbaki,}
\newblock{\em \'El\'ements de Math\'ematique. Groupes et Alg\`ebres de Lie, Chapitres 4,5, et 6,}
\newblock Masson, Paris (1981). 

\bibitem{survey_Tim} T. C. Burness, Simple groups, generation and probabilistic methods. Groups St Andrews 2017 in Birmingham 455 (2019): 200--229.

\bibitem{bgh}
T. Burness, R. Guralnick and S. Harper, {The spread of
	a finite group}, Ann. of Math. {193} (2021), no.~2, 619--687.
	
\bibitem{Carter}
R.W. Carter, Simple groups of Lie type, Wiley Classics Library, Reprint of the 1972 original,
              A Wiley-Interscience Publication, John Wiley \& Sons, Inc., New York, 1989.
              
\bibitem{Carter2}
R.W. Carter, Finite groups of {L}ie type:Conjugacy classes and complex characters, Pure and Applied Mathematics (New York),
A Wiley-Interscience Publication, John Wiley \& Sons, Inc., New York, 1985.
              

\bibitem{CL3} E. Crestani and A. Lucchini, The generating graph of finite soluble groups. Israel Journal of Mathematics 198 (2013), no. 1, 63--74.

\bibitem{survey_Scott} S. Harper, The spread of finite and infinite groups, arXiv:2210.09635 (2022).

\bibitem{comp} A. Lucchini, F. Menegazzo and M. Morigi,  On the existence of a complement for a finite simple group in its automorphism group. Special issue in honor of Reinhold Baer (1902–1979). Illinois J. Math. 47 (2003), no. 1-2, 395–418. 

\bibitem{Steinberg}
R.~Steinberg, \textit{Lectures on Chevalley groups}. Mimeographed notes,
Department of Math., Yale University, 1967. Now available as vol.~66 of
the University Lecture Series, Amer. Math. Soc., Providence, R.I., 2016.

\bibitem{Taylor}
D.E. Taylor,The geometry of the classical groups, Sigma Series in Pure Mathematics 9,
          Heldermann Verlag, Berlin, 1992.

\bibitem{WZ} Y. Wei and Y.M. Zou, Inverses of {C}artan matrices of {L}ie algebras and {L}ie
              superalgebras. Linear Algebra Appl. 521 (2017), 283--298. 




\end{thebibliography}
\end{document}